\documentclass[12pt]{article}
\usepackage{amsthm}
\usepackage{amsmath}
\usepackage{enumerate}
\usepackage{amssymb}
\usepackage{latexsym}
\usepackage{amsfonts}
\usepackage{color}
\usepackage{secdot}
\usepackage{mathrsfs}
\usepackage{subfigure}
\usepackage{epsfig}
\usepackage{natbib}
\newtheorem{theorem}{\bf Theorem}[section]

\newtheorem{definition}[theorem]{\bf Definition}
\newtheorem{corollary}{Corollary}[section]

 \newtheorem{counterexample}{Counterexample}[section]
\newtheorem{example}{Example}[section]

\newtheorem{lemma}{Lemma}[section]

 \newtheorem{assumption}{Assumption}[section]
 \newtheorem{set-up}{Set-up:}[section]
\begin{document}
\title{Ordering results of extreme order statistics from multiple-outlier scale models with dependence}
\author{{\large { Sangita {\bf Das}\thanks {Email address (corresponding author):
                sangitadas118@gmail.com}~ and Suchandan {\bf
                Kayal}\thanks {Email address :
                kayals@nitrkl.ac.in,~suchandan.kayal@gmail.com}}} \\
    { \em \small {\it Department of Mathematics, National Institute of
            Technology Rourkela, Rourkela-769008, India.}}}
\date{}
\maketitle
%\begin{center}
    \noindent{\bf Abstract:}
    In this paper, we focus on stochastic comparisons of extreme order statistics stemming from multiple-outlier scale models with dependence. Archimedean copula is used to model dependence structure among nonnegative random variables. Sufficient conditions are obtained for comparison of the largest order statistics in the sense of the usual stochastic, reversed hazard rate, star and Lorenz orders. The smallest order statistics are also compared with respect to the usual stochastic, hazard rate, star and Lorenz orders. To illustrate the theoretical establishments, some examples are provided.
%\end{center}

\section{Introduction }
Order statistics play a vital role in many fields such as statistical inference, economics, reliability theory and operations research. Consider a random sample $X_{1},\cdots,X_{n}$ from a population. Then, the $i$th order statistic is denoted by $X_{i:n}$, where $i=1,\cdots,n.$ In reliability theory, the $i$th order statistic represents the lifetime of a $(n-i+1)$-out-of-$n$ system, which functions if at least $n-i+1$ of $n$ components work. In particular, the order statistics $X_{1:n}$ and $X_{n:n}$ represent the lifetimes of series and parallel systems, respectively. Due to the correspondence between the order statistics and the systems' reliability, a lot of effort has been put to study ordering results between order statistics in terms of many well known stochastic orders. In this paper, we deal with the comparison of extreme order statistics arising from dependent multiple-outlier scale models in the sense of the usual stochastic, reversed hazard rate, hazard rate, star and Lorenz orders.
A multiple-outlier model is a collection of random variables
$X_1,\cdots,X_n$ such that $X_i\overset{st}{=} X,~ i =1,\cdots, p$ and $X_i=Y,~ i = p+1,\cdots, n,$ where $ 1 \leq p < n$.
Here, the notation $X_i\overset{st}{=} X$ means that the distributions of $X_i$ and $X$ are same. Due to the robustness of
different estimators of model parameters, multiple-outlier models
have been widely used by many researchers. Now, we present some developments on stochastic comparisons between order statistics arising from multiple-outlier models.
%In this regard, one may refer to \cite{gross1986}, \cite{balakrishnan1994}, \cite{barnett1994}, \cite{khaledispace2001}, \cite{hu2006},  \cite{xu2007} and \cite{balakrishnan2007permanents}.
 \cite{kochar2011} considered multiple-outlier exponential models. They showed that more heterogeneity among the scale parameters of the model results more skewed order statistics. \cite{zhao2012} took similar model and obtained ordering results between the largest order statistics with respect to the likelihood ratio, reversed hazard rate, hazard rate and usual stochastic orderings.
\cite{zhao2015comparison} discussed stochastic comparisons of the largest order statistics
from multiple-outlier gamma models in terms of various stochastic orderings such as the likelihood ratio, hazard rate, star and dispersive orders. \cite{kochar2015stochastic} established likelihood ratio ordering between the largest order statistics arising
from independent multiple-outlier scale models. Sufficient conditions
for the comparison of the lifetimes of series systems with respect to dispersive order have been obtained by \cite{fang2016}. They considered that the components of the series systems follow multiple-outlier Weibull models. \cite{aminiSeresht2016} studied multiple-outlier proportional hazard rate models and developed ordering results with respect to the star, Lorenz and dispersive orders. Further, they proved that more heterogeneity among the multiple-outlier components led to a more skewed lifetime of a $k$-out-of-$n$ system consisting of these components. \cite{wang2017} studied an open problem on mean residual life ordering between two parallel systems under multiple-outlier exponential models which was proposed by \cite{balakrishnan2013ordering}.  \\

Let $\{X_1,\cdots,X_{p},X_{p+1},\cdots, X_n\}$ be a collection of $n$ independent random variables following the multiple-outlier exponential model, where $X_{i},$ $i=1,\cdots,p$ follow exponential distribution with parameter $\lambda_1$ and $X_{j},$ $j=p+1,\cdots,n$ follow exponential distribution with parameter $\lambda_2$, with $n=p+q.$  Further, let $\{Y_1,\cdots,Y_{p^*},Y_{p^*+1},\cdots, Y_{n^{*}}\}$ be a collection of $n^*$ independent random variables following the multiple-outlier exponential model, where $Y_{i},$ $i=1,\cdots,p^*$ follow exponential distribution with parameter $\lambda_1^*$ and $Y_{j},$ $j=p^*+1,\cdots,n^*$ follow exponential distribution with parameter $\lambda_2^*$, with $n^*=p^*+q^*.$ Denote the largest order statistics by $X_{n:n}(p, q)$ and  $ Y_{n^*:n^*}(p^*, q^*)$ arising from $\{X_1,\cdots,X_{p},X_{p+1},\cdots, X_n\}$ and $\{Y_1,\cdots,Y_{p^*},Y_{p^*+1},\cdots, Y_{n^{*}}\}$, respectively. Under this set-up, \cite{balakrishnan2016comparisons} obtained conditions under which the likelihood ratio order holds between $X_{n:n}(p, q)$ and  $ Y_{n^*:n^*}(p^*, q^*)$. In particular, for $(p^*,q^*)\succeq^{w}(p,q)$, they showed that
\begin{align}\label{r1}
	(\underbrace{\lambda_{1},\cdots,\lambda_{1}}_{p},\underbrace{\lambda_{2},\cdots,\lambda_{2}}_{q})\succeq^{w} (\underbrace{\lambda^{*}_{1},\cdots,\lambda^{*}_{1}}_{p},\underbrace{\lambda^{*}_{2},\cdots,\lambda^{*}_{2}}_{q})
	\Rightarrow X_{n:n}(p, q)\geq_{lr} Y_{n^*:n^*}(p^{*}, q^{*}),
\end{align}
when $\lambda_1\leq\lambda^{*}_1\leq\lambda^{*}_{2}\leq\lambda_{2}$ and $1\leq p^*\leq p\leq q\leq q^*.$ The authors also extended the result given by \eqref{r1} to the case of the proportional hazard rate models.  Recently, \cite{torrado2017stochastic} developed the comparison result similar to \eqref{r1} for the multiple-outlier scale models when the random variables are independent. It is noted that almost all concerned research in this area has been developed under the assumption of statistically independent component lifetimes. However, there are some practical situations, where the condition of statistically mutual independence among the component lifetimes is evidently unsuitable. For an example, let us consider a mechanical system. The components of the system are suffering a common stress. Then, it is of huge interest to include statistical dependence among component lifetimes into the study of stochastic comparison of the lifetimes of the series and parallel systems. Further, note that due to complexity of working with the dependent random variables, marginal effort was put to the study of dependent multiple-outlier models by the researchers (see \cite{navarro2018}). These are the main motivations to investigate ordering properties of the extreme order statistics arising from multiple-outlier dependent scale components. The dependency structure among the random variables is modeled by the concept of Archimedean copulas.
We recall that a nonnegative random variable $X$ with distribution function $F_X$ is said to follow the scale model if there exists $\lambda>0$ such that $F_X(x)=F(\lambda x)$, where  $F$ is  the baseline distribution function and $\lambda$ is the scale parameter.\\

%For recent develpements on this topic, we refer \cite{zhao2015}, \cite{hazra2017stochastic}, \cite{hazra2018stochastic}, \cite{chowdhury2019}, \cite{kayal2019ls}, \cite{kayal2019mo}, \cite{das2020}, \cite{barmalzan2020}, \cite{kundu2020}. Motivated from the literature, the study of multiple-outlier scale model having dependent structure is more interesting to the researchers than the independent case due to its complexity.

In this paper, we will develop different ordering results between the largest as well as the smallest order statistics stemming from multiple-outlier dependent scale models with respect to several stochastic orderings such as the usual stochastic, hazard rate, reversed hazard rate, star and Lorenz orders. Let $\{X_1,\cdots,X_{n^{*}_1},X_{{n^{*}_1}+1},\cdots, X_{n^{*}}\}$  be a set of dependent and heterogeneous random observations. The observations are sharing a common Archimedean copula with generator $\psi_1$ and are taken from the multiple-outlier scale model, where for $i=1,\cdots,n^{*}_{1},$ $X_{i}\sim F_{1}(\lambda_1 x)$  and for $j=n^{*}_{1}+1,\cdots,n^{*},$  $X_{j}\sim F_{2}(\lambda_2 x)$, where $\lambda_1,~\lambda_2>0$.  Note that $F_1(.)$ and $F_{2}(.)$ are two different baseline distribution functions. Also, let $\{Y_1,\cdots,Y_{n^{*}_1},Y_{{n^{*}_1}+1},\cdots, Y_{n^{*}}\}$ be another set of dependent and heterogeneous random observations sharing a common Archimedean copula with generator $\psi_2,$ drawn from the multiple-outlier scale model, where for $i=1,\cdots,n^{*}_{1},$ $Y_{i}\sim F_{1}(\mu_1 x)$ and for $j=n^{*}_{1}+1,\cdots,n^{*},$  $Y_{j}\sim F_{2}(\mu_2 x)$, where $\mu_1,~\mu_2>0$.
%let $Y_{i}\sim S(\mu_{1},F_1,\psi_1)$ and for $j=n_{1}+1,\cdots,n, $ $Y_{j}\sim S(\mu_{2},F_2,\psi_2).$ A non negative random variable $X$ is said to follows scale family (denoted by $X\sim S(\lambda,F)$) with baseline distribution function $F$, if the cumulative distribution function of $X$ can be written as
%$F_X(x)=F(x\lambda),$ where $x,~\lambda>0.$
Denote by $r_1,~\tilde{r}_1$ and $r_2,~\tilde{r}_2$ the hazard rate and reversed hazard rate functions for $F_1$ and $F_2,$ respectively.
Further,
denote $X_{n:n}(n_{1},n_{2})$, $Y_{n^{*}:n^{*}}(n^{*}_{1},n^{*}_{2})$ and $X_{1:n}(n_{1},n_{2})$,  $Y_{1:n^{*}}(n^{*}_{1},n^{*}_{2})$ are the largest and the smallest order statistics, respectively arising  from $\{X_1,\cdots,X_{n_1},
X_{{n_1}+1},\cdots, X_{n}\}$ and  $\{Y_1,\cdots,Y_{n^{*}_1},Y_{{n^{*}_1}+1},\cdots, Y_{n^{*}}\}$, where $1\leq n_{1}\leq n^{*}_{1}\leq n^{*}_{2}\leq n_{2}$, $n=n_1+n_2$ and $n^*=n^*_1+n^*_2.$ We aim to establish sufficient conditions, under which the following implications hold:
\begin{eqnarray*}
(\underbrace{\lambda_{1},\cdots,\lambda_{1}}_{n^*_{1}},\underbrace{\lambda_{2},\cdots,\lambda_{2}}_{n^*_{2}})\succeq^{w}(\underbrace{\mu_{1},\cdots,\mu_{1}}_{n^{*}_{1}},\underbrace{\mu_{2},\cdots,\mu_{2}}_{n^{*}_{2}})&\Rightarrow& Y_{n^{*}:n^{*}}(n^{*}_{1},n^{*}_{2})\leq_{st}[\leq_{rh}]X_{n:n}(n_{1},n_{2}),\\
(\underbrace{\lambda_{1},\cdots,\lambda_{1}}_{n^*_{1}},\underbrace{\lambda_{2},\cdots,\lambda_{2}}_{n^*_{2}})\succeq_{w}(\underbrace{\mu_{1},\cdots,\mu_{1}}_{n^{*}_{1}},\underbrace{\mu_{2},\cdots,\mu_{2}}_{n^{*}_{2}})&\Rightarrow& X_{1:n}(n_{1},n_{2})\leq_{st}Y_{1:n^{*}}(n^{*}_{1},n^{*}_{2})
\end{eqnarray*}
and $$(\underbrace{u_{1},\cdots,u_{1}}_{n^{*}_{1}},\underbrace{u_{2},\cdots,u_{2}}_{n^{*}_{2}})\succeq_{w}(\underbrace{v_{1},\cdots,v_{1}}_{n^{*}_{1}},\underbrace{v_{2},\cdots,v_{2}}_{n^{*}_{2}})
\Rightarrow X_{1:n}(n_{1},n_{2})\leq_{hr}Y_{1:n^*}(n^{*}_{1},n^{*}_{2}),$$
where $u_{i}=\log\lambda_i$ and $v_{i}=\log\mu_i$, $i=1,~2.$\\

The remainder of the paper is rolled out as follows. Some basic definitions and important lemmas are provided in Section $2$. Section $3$ consists of two subsections. In Subsection $3.1$, we obtain sufficient conditions, under which two largest order statistics are comparable  according to the usual stochastic order, reversed hazard rate order, star order and Lorenz order, whereas in Subsection $3.2$, we study the usual stochastic order, hazard rate order, star order and Lorenz order between two smallest order statistics. We also present some examples to illustrate the established results. Finally, we conclude the paper in Section $4$.\\

Throughout the paper, we only concern about nonnegative random variables. Increasing and decreasing mean nondecreasing and nonincreasing, respectively. Also, the prime `$\prime$' stands for the first order derivative.
\section{Basic notions\setcounter{equation}{0}}
In this section, we recall some basic definitions and well known concepts of stochastic orders and majorization. Let  $\boldsymbol{x} =
\left(x_{1},\cdots,x_{n}\right)$ and $\boldsymbol{y} =
\left(y_{1},\cdots,y_{n}\right)$ be two $n$-dimensional vectors such that $\boldsymbol{x}~,\boldsymbol{y}\in\mathbb{A}$, where $\mathbb{A} \subset \mathbb{R}^{n}$ and $\mathbb{R}^{n}$ is
an $n$-dimensional Euclidean space. Also, consider the order coordinates of the vectors $\boldsymbol{x}$ and $\boldsymbol{y}$ as  $x_{1:n}\leq \cdots \leq x_{n:n}$ and
$y_{1:n}\leq\cdots \leq y_{n:n},$  respectively.
\begin{definition}\label{definition2.2}
	A vector $\boldsymbol{x}$ is said to be
	\begin{itemize}
		\item  majorized by another vector $\boldsymbol{y},$ (denoted
		by $\boldsymbol{x}\preceq^{m} \boldsymbol{y}$), if for each $l=1,\cdots,n-1$, we have
		$\sum_{i=1}^{l}x_{i:n}\geq \sum_{i=1}^{l}y_{i:n}$ and
		$\sum_{i=1}^{n}x_{i:n}=\sum_{i=1}^{n}y_{i:n};$
		
		\item weakly submajorized by another vector $\boldsymbol{y},$ denoted
		by $\boldsymbol{x}\preceq_{w} \boldsymbol{y}$, if for each $l=1,\cdots,n$, we have
		$\sum_{i=l}^{n}x_{i:n}\leq \sum_{i=l}^{n}y_{i:n};$
		
		\item weakly supermajorized by  another vector $\boldsymbol{y},$ denoted
		by $\boldsymbol{x}\preceq^{w} \boldsymbol{y}$, if for each $l=1,\cdots,n$, we have
		$\sum_{i=1}^{l}x_{i:n}\geq \sum_{i=1}^{l}y_{i:n}.$
		
%		\item $p$-larger than the vector $\boldsymbol{y}$, denoted by
%		$\boldsymbol{x}\succeq^{p}\boldsymbol{y}$ if $\prod_{i=1}^{l}x_{i:n} \le
%		\prod_{i=1}^{l}y_{i:n}$ for $l=1,\ldots,n.$
		
%		\item reciprocally majorized by another vector $\boldsymbol{y},$ denoted
%		by $\boldsymbol{x}\preceq^{rm}\boldsymbol{y}$, if $\sum_{i=1}^{l}x_{i:n}^{-1}\le
%		\sum_{i=1}^{l}y_{i:n}^{-1}$ for all $l=1,\ldots,n$.
%	
\end{itemize}
\end{definition}
Note that  $\boldsymbol{x}\preceq^{m} \boldsymbol{y}$ implies both  $\boldsymbol{x}\preceq_{w} \boldsymbol{y}$ and  $\boldsymbol{x}\preceq^{w} \boldsymbol{y}.$ For brief introduction of majorization orders and their applications, we refer to \cite{Marshall2011}.
 Now, we present notions of stochastic orderings. Let $X_1$ and $X_2$ be two nonnegative random variables with
probability density functions (PDFs) $f_{X_1}$ and $f_{X_2}$, cumulative density functions (CDFs) $F_{X_1}$ and $F_{X_2}$, survival functions $\bar
F_{X_1}=1-F_{X_1}$ and $\bar F_{X_2}=1-F_{X_2}$, hazard rate functions  $r_{X_1}=f_{X_1}/\bar
F_{X_1}$ and  $r_{X_2}=f_{X_2}/
\bar F_{X_2}$ and reversed hazard rate functions $ \tilde r_{X_1}=f_{X_1}/
F_{X_1}$ and $\tilde r_{X_2}=f_{X_1}/
F_{X_2}$, respectively.

%\subsection{Stochastic orders}

\begin{definition}
    A random variable $X_1$ is said to be smaller than $X_2$ in the
    \begin{itemize}
%        \item likelihood ratio order (denoted by $X_1\leq_{lr}X_2$) if
%        $f_{X_2}(x)/f_{X_1}(x)$ is increasing in $x>0$,
        \item hazard rate order (denoted by $X_1\leq_{hr}X_2$)
        if $r_{X_1}(x)\geq r_{X_2}(x)$, for all $x>0$;
        \item reversed hazard rate order (denoted by $X_1\leq_{rh}X_2$)
        if $\tilde r_{X_1}(x)\leq \tilde r_{X_2}(x)$, for all $x>0$;
        \item usual stochastic order (denoted by $X_1\leq_{st}X_2$) if
        $\bar F_{X_1}(x)\leq\bar F_{X_2}(x)$, for all $x$;
         \item star order (denoted by $X_{1}\leq_{*}X_{2}$ or $F_{X_1}(x)\leq_{*}F_{X_2}(x)$) if  $F^{-1}_{X_2}F_{X_1}(x)$
        is star shaped in the sense that $\frac{F^{-1}_{X_2}F_{X_1}(x)}{x}$ is increasing in $x$ on the support of $X_1$;
        \item Lorenz order (denoted by $X_1\leq_{Lorenz}X_2$) if
        $$\frac{1}{E(X_1)}\int_{0}^{F^{-1}_{X_1}(u)}x dF_{X_1}(x)\geq\frac{1}{E(X_2)}\int_{0}^{F^{-1}_{X_2}(u)}x dF_{X_2}(x),\text{ for all } u\in (0,1].$$
    %   \item ageing faster order in terms of hazard
    %   rate order (denoted by $X\leq_{R-hr}$Y) if $r_X(x)/r_Y(x)$ is
    %   increasing in $x$ for which the ratio is well defined, for all $x$.
    \end{itemize}
\end{definition}
Note that both the hazard rate and reversed hazard rate orderings imply the usual stochastic ordering. Also, star order implies Lorenz order (see \cite{marshall2007life}). One may refer to \cite{shaked2007stochastic} for a detailed discussion on various stochastic orderings. The next definition is for the Schur-convex and
Schur-concave functions.
\begin{definition}
    A function $\Psi:\mathbb{R}^n\rightarrow \mathbb{R}$ is said
    to be Schur-convex (Schur-concave) in $\mathbb{R}^n$ if
    $$\boldsymbol {x}\overset{m}{\succeq}\boldsymbol{ y}\Rightarrow
    \Psi(\boldsymbol { x})\geq( \leq)\Psi(\boldsymbol { y}) \text{, for all } \boldsymbol { x},~ \boldsymbol
    { y} \in \mathbb{R}^n.$$
\end{definition}
 Throughout the article, we will use the notations.
$(i)~\mathcal{D}_{+}=\{(x_1,\cdots,x_n):x_{1}\geq
x_{2}\geq\cdots\geq x_{n}>0\}$ and $(ii)~\mathcal{E}_{+}=\{(x_1,\cdots,x_n):0<x_{1}\leq
x_{2}\leq\cdots\leq x_{n}\}.$
 Denote by $h'( z)=\frac{d h(z)}{d z}.$
The following consecutive lemmas due to  \cite{kundu2016some} are useful to prove the results in the subsequent sections. The partial derivative of $h$ with respect to its $k$th argument is denoted by $h_{(k)}(\boldsymbol{ z})=\partial h(\boldsymbol{z})/\partial z_k,$ for $k=1,\cdots,n.$

{\begin{lemma}\label{lem2.1a}
		Let $h:\mathcal{D_+}\rightarrow \mathbb{ R}$ be a function, continuously differentiable on the interior of $\mathcal{D_+}.$ Then, for $\boldsymbol{x},~\boldsymbol{y}\in\mathcal{D_+},$
		$$\boldsymbol{x}\succeq^{m}\boldsymbol{y}\text{ implies } h(\boldsymbol{x})\geq(\leq )~h(\boldsymbol{y}),$$
		if and only if
		$h_{(k)}(\boldsymbol{ z}) \text{ is decreasing (increasing) in } k=1,\cdots,n.$
	\end{lemma}
	\begin{lemma}\label{lem2.1b}
		Let $h:\mathcal{E_+}\rightarrow \mathbb{ R}$ be a function, continuously differentiable on the interior of $\mathcal{E_+}.$ Then, for $\boldsymbol{x},~\boldsymbol{y}\in\mathcal{E_+},$
		$$\boldsymbol{x}\succeq^{m}\boldsymbol{y}\text{ implies } h(\boldsymbol{x})\geq(\leq )~h(\boldsymbol{y}),$$
		if and only if
		$h_{(k)}(\boldsymbol{ z}) \text{ is increasing (decreasing) in } k=1,\cdots,n.$
	\end{lemma}

The following lemma due to \cite{saunders1978} is useful to establish star order between the order statistics. 
\begin{lemma}\label{lem4}
	Let $\{F_{\lambda}|\lambda\in\mathbb{R}\}$	be a class of distribution functions, such that $F_{\lambda}$ is supported on some interval $(a,b)\subseteq(0,\infty)$ and has density $f_{\lambda}$ which does not vanish on any subinterval of $(a,b)$. Then,
	$$F_{\lambda}\leq_{*} F_{\lambda^*},~~~\lambda\leq \lambda^*$$
	if and only if
	$$\frac{F'_{\lambda}(x)}{xf_{\lambda}(x)}\text{ is decreasing in }x,$$ where $F'_{\lambda}$ is the derivative of $F_{\lambda}$ with respect to $\lambda.$
\end{lemma}

To model the dependency structure among the random variables, the concept of copulas plays a vital role. One of the important characteristics of the copula is that it involves the information of the dependencies between the random variables apart from the behavior of the marginal distributions. Archimedean copulas are important class of copulas. These are used widely because of its simplicity.
Let $F$ and $\bar F$ be the joint distribution function and the joint survival function of the random vector $\boldsymbol{X}=(X_1,\cdots,X_n)$. Suppose there exist functions $C(\boldsymbol{z}):[0,1]^n\rightarrow [0,1]$ and
$\hat {C}(\boldsymbol{z}):[0,1]^n\rightarrow [0,1]$ such that for all $ x_i,~i\in \mathcal I_n, $ where $\mathcal I_n$ is the index set
$$ F(x_1,\cdots,x_n)=C(F_1(x_1),\cdots,F_n(x_n))$$
and
$$\bar{F}(x_1,\cdots,x_n)=\hat{C}(\bar{F_1}(x_1),\cdots,\bar{F_n}(x_n))$$ hold, where $\boldsymbol{z}=(z_1,\cdots,z_n)$.
Then, $C(\boldsymbol{z})$ and $\hat{C}(\boldsymbol{z})$ are said to be the  copula and survival copula of $\boldsymbol{X}$, respectively. Here, $F_1,\cdots,F_n$ and $\bar{F_1},\cdots,\bar{F_n}$ are the univariate marginal distribution functions and survival functions of the random variables $X_1,\cdots,X_n$, respectively.\\
Now, let $\psi:[0,\infty)\rightarrow[0,1]$ be a nonincreasing and continuous function, satisfying $\psi(0)=1$ and $\psi(\infty)=0.$ Also, let $\psi={\phi}^{-1}=\text{sup}\{x\in \mathcal R:\phi(x)>v\}$ be the right continuous inverse. Further, suppose $\psi$ satisfies the conditions $(i)$ $(-1)^i{\psi}^{i}(x)\geq 0,~ i=0,~1,\cdots,d-2$ and  $(ii)$ $(-1)^{d-2}{\psi}^{d-2}$ is nonincreasing and convex. That implies the generator $\psi$ is $d$-monotone. Then, a copula $C_{\psi}$ is said to be an Archimedean copula if it can be written as the following form $$C_{\psi}(v_1,\cdots,v_n)=\psi({\phi(v_1)},\cdots,\phi(v_n)),~\text{ for all } v_i\in[0,1],~i\in\mathcal{I}_n.$$ For further discussion on Archimedean copulas, one may refer to \cite{nelsen2006introduction} and \cite{mcneil2009multivariate}.
% A list of some well-known Archimedean copulas that could be used in sequal section is presented in Table \ref{table1}.
 
Next lemma is taken from \cite{li2015ordering}, which has been used to prove the results in Theorems \ref{th1}, \ref{th4}, \ref{th2} and \ref{th6}.
\begin{lemma}\label{lem3.1}
	For two $n$-dimensional Archimedean copulas $C_{\psi_1}$ and  $C_{\psi_2}$, if $\phi_2\circ\psi_1$ is super-additive, then $C_{\psi_1}(\boldsymbol{z})\leq C_{\psi_2}(\boldsymbol{z})$, for all $\boldsymbol{z}\in[0,1]^n.$ A function $f$ is said to be super-additive, if $ f(x)+f(y)\leq f(x+y),$ for all $x$ and $y$ in the domain of $f.$
\end{lemma}

\section{Main Results\setcounter{equation}{0}}
This section is completely devoted to establish sufficient conditions, under which the extreme order statistics arising from multiple outlier dependent scale models are comparable in different stochastic senses. The usual stochastic, hazard rate, reversed hazard rate,  star and Lorenz orders are used in this sequel.
Throughout this section, we denote two dimensional vectors by bold symbols. For example, $\boldsymbol{\lambda}=(\lambda_{1},\lambda_{2})$ and         $\boldsymbol{\mu}=(\mu_{1},\mu_{2}).$

\subsection{Orderings between the largest order statistics}
This subsection addresses ordering results between the largest order statistics arising from multiple-outlier models.  The following three consecutive theorems present different conditions, for which the usual stochastic order between the largest order statistics holds. Before presenting the first result, we state the following assumption.
\begin{assumption}\label{ass1}
	Let $X_{1},\cdots,X_{n^{*}}$ be  $n^*$ dependent nonnegative random variables sharing Archimedean copula with generator $\psi_1,$ with $X_{i}\sim F_1(x\lambda_1),$ for $i=1,\cdots,n^*_1$ and  $X_{j}\sim F_2(x\lambda_2),$ for $j={n^*_{1}+1},\cdots,n^*$. Also, 	let $Y_{1},\cdots,Y_{n^{*}}$ be  $n^*$ dependent non-negative random variables sharing Archimedean copula with generator $\psi_2,$ with $Y_{i}\sim F_1(x\mu_1),$ for $i=1,\cdots,n^*_1$ and $Y_{j}\sim F_2(x\mu_2),$ for $j={n^*_{1}+1},\cdots,n^*.$ Here, $n^{*}_{1}+n^{*}_{2}=n^{*}$, $\psi_{1}=\phi^{-1}_{1}$ and $\psi_{2}=\phi^{-1}_{2}$.
\end{assumption}

\begin{theorem}\label{th1}
 Under the set-up as in Assumption \ref{ass1}, let $\tilde{r}_1(x)\geq(\leq)\tilde{r}_2(x)$ and $n^{*}_1\geq(\leq)n^{*}_2$. Then,
 $$(\underbrace{\lambda_{1},\cdots,\lambda_{1}}_{n^{*}_{1}},\underbrace{\lambda_{2},\cdots,\lambda_{2}}_{n^{*}_{2}})\succeq^{w}
 (\underbrace{\mu_{1},\cdots,\mu_{1}}_{n^{*}_{1}},\underbrace{\mu_{2},\cdots,\mu_{2}}_{n^{*}_{2}})\Rightarrow Y_{n^{*}:n^{*}}(n^{*}_{1},n^{*}_{2})\leq_{st}X_{n^{*}:n^{*}}(n^{*}_{1},n^{*}_{2}),$$ provided $\boldsymbol{ \lambda},~\boldsymbol{\mu}\in\mathcal{E}_+~(\mathcal{D}_+),$ $\phi_2\circ\psi_1$ is super-additive,
 $\psi_1\text{ or }\psi_2$ is log-convex and  $\tilde{r}_1(x)\text{ or }\tilde{r}_2(x)$ is decreasing.
\end{theorem}

\begin{proof}
The distribution functions of $X_{n^{*}:n^{*}}(n^{*}_1,n^{*}_2)$ and $Y_{n^{*}:n^{*}}(n^{*}_1,n^{*}_2)$ are respectively given by
$$F_{X_{n^{*}:n^{*}}}(n^{*}_1,n^{*}_2)(x)=\psi_1\left[n^{*}_1\phi_1\left(F_1\left(x\lambda_1\right)\right)+n^{*}_2\phi_1\left(F_2\left(x\lambda_2\right)\right) \right]$$ and
$$F_{Y_{n^{*}:n^{*}}}(n^{*}_1,n^{*}_2)(x)=\psi_2\left[n^{*}_1\phi_2\left(F_1\left(x\mu_1\right)\right)+n^{*}_2\phi_2\left(F_2\left(x\mu_2\right)\right) \right].$$
Denote $A(\boldsymbol{ \lambda},\psi_1,x)=F_{X_{n^{*}:n^{*}}}(n^{*}_1,n^{*}_2)(x)$ and $B(\boldsymbol{ \mu},\psi_2,x)=F_{Y_{n^{*}:n^{*}}}(n^{*}_1,n^{*}_2)(x).$ Using the fact that $\phi_2\circ\psi_1$ is super-additive, one can easily obtain
$A(\boldsymbol{ \mu},\psi_1,x)\leq B(\boldsymbol{ \mu},\psi_2,x).$ Therefore, to prove the desired result, we have to show that
 $A(\boldsymbol{ \lambda},\psi_1,x)\leq A(\boldsymbol{ \mu},\psi_1,x).$ This is equivalent to establish that the function $A(\boldsymbol{ \lambda},\psi_1,x)$ is increasing and Schur-concave with respect to $\boldsymbol{ \lambda}$ (see Theorem $A.8$ of \cite{Marshall2011}). Further,
on differentiating $A(\boldsymbol{ \lambda},\psi_1,x)$ with respect to $\lambda_1$ partially, we get
\begin{equation}\label{eq1}
\frac{\partial A(\boldsymbol{ \lambda},\psi_1,x)}{\partial\lambda_1}=xn^{*}_1{\tilde{r}_1(x\lambda_{1})}\frac{{\psi_1}\left[\phi_1\left[F_1\left(x\lambda_1\right)\right]\right]}{{\psi}'_1\left[\phi_1\left[F_1\left(x\lambda_1\right)\right]\right]}{\psi}'_1\left[n^{*}_1\phi_1\left(F_1\left(x\lambda_1\right)\right)+n^{*}_2\phi_1\left(F_2\left(x\lambda_2\right)\right) \right].
\end{equation}
 From \eqref{eq1}, it is not difficult to check that $\frac{\partial A(\boldsymbol{ \lambda},\psi_1,x)}{\partial\lambda_1}\geq 0.$ Similarly, $\frac{\partial A(\boldsymbol{ \lambda},\psi_1,x)}{\partial\lambda_2}\geq 0.$ Thus, $A(\boldsymbol{ \lambda},\psi_1,x)$ is increasing in $\lambda_i$, for $i=1,2.$ To establish Schur-concavity of $A(\boldsymbol{ \lambda},\psi_1,x),$ in view of Lemma \ref{lem2.1b} (Lemma \ref{lem2.1a}), we only need to show that
 for  $1\leq i\leq j\leq n^{*},$ the following inequality holds:
	\begin{equation}\label{eq3.2}
	\frac{\partial A(\boldsymbol{ \lambda},\psi_1,x)}{\partial\lambda_i}-\frac{\partial A(\boldsymbol{ \lambda},\psi_1,x)}{\partial\lambda_j}\geq(\leq) 0\text{, for }\boldsymbol{ \lambda}\in\mathcal{E}_+~(\mathcal{D}_+).
	\end{equation}
Next, consider three cases.\\
Case I: For $1\leq i\leq j\leq n^{*}_1 ,$ $\lambda_i=\lambda_j=\lambda_1.$ In this case, $\frac{\partial A(\boldsymbol{ \lambda},\psi_1,x)}{\partial\lambda_i}-\frac{\partial A(\boldsymbol{ \lambda},\psi_1,x)}{\partial\lambda_j}=0.$\\
Case II: For $n^{*}_{1}+1\leq i\leq j\leq n^{*} ,$  $\lambda_i=\lambda_j=\lambda_2.$ Here, $\frac{\partial A(\boldsymbol{ \lambda},\psi_1,x)}{\partial\lambda_i}-\frac{\partial A(\boldsymbol{ \lambda},\psi_1,x)}{\partial\lambda_j}=0.$\\
Case III: For $1\leq i\leq n^{*}_1 $ and $n^{*}_{1}+1\leq j\leq n^{*},$  $\lambda_i=\lambda_1$ and $\lambda_j=\lambda_2$.
For this case, consider $\lambda_1\leq(\geq)\lambda_2$, which implies $\phi_{1}(F_1(x\lambda_1))\geq(\leq)\phi_{1}(F_1(x\lambda_2)).$ Further, under the given assumption, we get  $\phi_{1}(F_1(x\lambda_2))\geq(\leq)\phi_{1}(F_2(x\lambda_2)).$ Hence, $\phi_{1}(F_1(x\lambda_1))\geq(\leq)\phi_{1}(F_2(x\lambda_2)).$ Again, $\psi_{1}$ is log-convex. Therefore, we have
\begin{equation}\label{eq2}
	-\frac{\psi_{1}(w)}{\psi_{1}'(w)}\Big|_{w=\phi_{1}[F_{1}(x\lambda_{1})]}\geq(\leq) -\frac{\psi_{1}(w)}{\psi_{1}'(w)}\Big|_{w=\phi_{1}[F_{2}(x\lambda_{2})]}.
\end{equation}
Moreover, $\tilde r_{1}(w)$ is decreasing in $w>0$, hence
\begin{equation}\label{eq3}
	\tilde r_{1}(x\lambda_{1})\geq(\leq)\tilde r_{1}(x\lambda_{2}).
\end{equation}
Also, $\tilde r_{1}(x)\geq(\leq)\tilde r_{2}(x)$ gives
\begin{equation}\label{eq3.}
\tilde r_{1}(x\lambda_{2})\geq(\leq)\tilde r_{2}(x\lambda_{2}).
\end{equation}
Equations \eqref{eq3}, \eqref{eq3.} and $n^{*}_1\geq(\leq)n^{*}_2$, together imply
\begin{equation}\label{eq3..}
	n^{*}_1\tilde r_{1}(x\lambda_{1})\geq(\leq)n^{*}_2\tilde r_{2}(x\lambda_{2}).
\end{equation}
Finally, combining \eqref{eq2} and \eqref{eq3..}, we obtain \eqref{eq3.2}. This completes the proof of the theorem.
\end{proof}
In the previous result, we assume that the dependence structures of two sets of samples having multiple-outliers are different. Also, first $n^{*}_1$ observations of  $\{X_1,\cdots,X_{n^{*}_1},X_{{n^{*}_1}+1},\cdots, X_{n^{*}}\}$ have baseline distribution function $F_1$ and remaining observations have baseline distribution function $F_{2}$. Similarly, for the other set of observations $\{Y_1,\cdots,Y_{n^{*}_1},Y_{{n^{*}_1}+1},\cdots, Y_{n^{*}}\}$. The following corollary, which is a direct consequence of Theorem \ref{th1} presents some special cases.
\begin{corollary}\label{cor1}
	In addition to Assumption \ref{ass1}, let $\psi_1=\psi_2=\psi$, $n^{*}_1\geq(\leq)n^{*}_2$ and $\psi$ be log-convex. Further, let $\boldsymbol{ \lambda},~\boldsymbol{\mu}\in\mathcal{E}_+~(\mathcal{D}_+)$. Then,
	\begin{itemize}
			\item [(i)]	
		 $(\underbrace{\lambda_{1},\cdots,\lambda_{1}}_{n^{*}_{1}},\underbrace{\lambda_{2},\cdots,\lambda_{2}}_{n^{*}_{2}})\succeq^{w}
		 (\underbrace{\mu_{1},\cdots,\mu_{1}}_{n^{*}_{1}},\underbrace{\mu_{2},\cdots,\mu_{2}}_{n^{*}_{2}})\Rightarrow Y_{n^{*}:n^{*}}(n^{*}_{1},n^{*}_{2})\leq_{st}X_{n^{*}:n^{*}}(n^{*}_{1},n^{*}_{2}),$ provided  $\tilde{r}_1(x)\text{ or }\tilde{r}_2(x)$ is decreasing and $\tilde{r}_1(x)\geq(\leq)\tilde{r}_2(x);$
		\item [(ii)] for $\tilde{r}_1=\tilde{r}_2=\tilde{r},$ we have
		 $$(\underbrace{\lambda_{1},\cdots,\lambda_{1}}_{n^{*}_{1}},\underbrace{\lambda_{2},\cdots,\lambda_{2}}_{n^{*}_{2}})\succeq^{w}
		 (\underbrace{\mu_{1},\cdots,\mu_{1}}_{n^{*}_{1}},\underbrace{\mu_{2},\cdots,\mu_{2}}_{n^{*}_{2}})\Rightarrow Y_{n^{*}:n^{*}}(n^{*}_{1},n^{*}_{2})\leq_{st}X_{n^{*}:n^{*}}(n^{*}_{1},n^{*}_{2}),$$ provided  $\tilde{r}(x)$ is decreasing.
	\end{itemize}

\end{corollary}
%\begin{corollary}\label{cor2}
%	Let Set-up \ref{ass1} hold. Then,
%	 $$(\underbrace{\lambda_{1},\cdots,\lambda_{1}}_{n_{1}},\underbrace{\lambda_{2},\cdots,\lambda_{2}}_{n_{2}})\succeq^{w}
%	(\underbrace{\mu_{1},\cdots,\mu_{1}}_{n_{1}},\underbrace{\mu_{2},\cdots,\mu_{2}}_{n_{2}})\Rightarrow Y_{n:n}(n_{1},n_{2})\leq_{st}X_{n:n}(n_{1},n_{2}),$$ provided $\boldsymbol{ \lambda},~\boldsymbol{\mu}\in\mathcal{E}_+(~\mathcal{D}_+),$
%	$\ln\psi$ is convex, $F_1(x)\leq(\geq)F_2(x)$, $\tilde{r}_1(x)\text{ or }\tilde{r}_2(x)$ is decreasing and $\tilde{r}_1(x)\geq(\leq)\tilde{r}_2(x)$.
%	\
%\end{corollary}

The next theorem states that the ordering result holds between the largest order statistics $X_{n:n}(n_{1},n_{2})$ and $X_{n^{*}:n^{*}}(n^{*}_{1},n^{*}_{2})$ according to the usual stochastic ordering. Here, the samples are  collected from multiple-outlier dependent scale models. Also, it is assumed that the samples are sharing Archimedean copula with a common generator.
\begin{assumption}\label{ass3}
	Let $X_{1},\cdots,X_{n^{*}}$ be $n^*$ dependent nonnegative random variables sharing Archimedean copula with generator $\psi_1,$ such that  $X_{i}\sim F_1(x\lambda_1)$,  for $i=1,\cdots,n^{*}_1$ and $X_{j}\sim F_{2}(x\lambda_{2}),$ for $j={n^{*}_{1}+1},\cdots,n^*.$ We assume that there exist two natural numbers $n_{1}$ and $n_{2}$ such that $1\leq n_{1}\leq n^{*}_{1}\leq n^{*}_{2}\leq n_{2}.$ Also,  $n=n_1+n_2,~n^*=n^*_1+n^*_2$ and $\psi_{1}=\phi^{-1}_{1}$.
\end{assumption}

\begin{theorem}\label{th3}
Let Assumption \ref{ass3} hold with $F_{1}\geq F_{2}$. Then, for  $\boldsymbol{\lambda}\in\mathcal{D_+}$,  we have
$$(n_{1},n_{2})\succeq_{w}(n^{*}_{1},n^{*}_{2})\Rightarrow X_{n^{*}:n^{*}}(n^{*}_{1},n^{*}_{2})\leq_{st}X_{n:n}(n_{1},n_{2}).$$
\end{theorem}

\begin{proof}
	The distribution functions of $X_{n:n}(n_1,n_2)$ and $X_{n^*:n^*}(n^*_1,n^*_2)$ can be written respectively as
\begin{equation*}
F_{X_{n:n}}(n_1,n_2)(x)=\psi_{1}[n_{1}\phi_{1}\left(F_{1}\left(x\lambda_1\right)\right)+n_{2}\phi_{1}\left(F_{2}\left(x\lambda_2\right)\right)]
\end{equation*}
and \begin{equation*}
F_{X_{n^{*}:n^{*}}}(n^*_1,n^*_2)(x)=\psi_{1}[n^{*}_{1}\phi_{1}\left(F_{1}\left(x\lambda_1\right)\right)+n^{*}_{2}\phi_{1}\left(F_{2}\left(x\lambda_2\right)\right)].
\end{equation*}
To obtain the desired result, one needs to show $F_{X_{n:n}}(n_1,n_2)(x)\leq F_{X_{n^*:n^*}}(n^*_1,n^*_2)(x)$. Equivalently, we have to establish that  $(n^{*}_{1}-n_{1})\phi_{1}\left(F_{1}\left(x\lambda_1\right)\right)\leq(n_{2}-n^{*}_{2})\phi_{1}\left(F_{2}\left(x\lambda_2\right)\right)$.
%	\begin{eqnarray}\label{eq9}
%	&F_{Y_{n:n}}(x)\geq F_{Y_{n^*:n^*}}(x)\nonumber\\
%		&\Rightarrow\psi_{2}\sum_{i=1}^{n}\phi_{2}\left(F_{i}\left(x\mu_i\right)\right)\geq \psi_{2}\sum_{i=1}^{n^{*}}\phi_{2}\left(F_{i}\left(x\mu_i\right)\right)\nonumber\\
%		&\Rightarrow n_{1}\phi_{2}\left(F_{1}\left(x\mu_1\right)\right)+n_{2}\phi_{2}\left(F_{2}\left(x\mu_2\right)\right)\leq n^{*}_{1}\phi_{2}\left(F_{1}\left(x\mu_1\right)\right)+n^{*}_{2}\phi_{2}\left(F_{2}\left(x\mu_2\right)\right)\nonumber\\
%		 &\Rightarrow(n_{1}-n^{*}_{1})\phi_{2}\left(F_{1}\left(x\mu_1\right)\right)\leq(n^{*}_{2}-n_{2})\phi_{2}\left(F_{2}\left(x\mu_2\right)\right).
%	\end{eqnarray}
Now, $(n_{1},n_{2})\succeq_{w}(n^{*}_{1},n^{*}_{2})\Rightarrow(n_{1}+n_{2})\geq(n^{*}_{1}+n^{*}_{2})\Rightarrow (n_{2}-n^{*}_{2})\geq(n^{*}_{1}-n_{1})\geq 0.$ Also, $\lambda_{1}\geq \lambda_2\Rightarrow \phi_{1}\left(F_{2}\left(x\lambda_2\right)\right)\geq\phi_{1}\left(F_1\left(x\lambda_1\right)\right)\geq 0.$
Using these arguments, we get the required inequality. Hence, the proof is completed.
\end{proof}
In Theorem \ref{th3}, if we take the same baseline distribution, then the following corollary is immediate.
\begin{corollary}\label{cor3}
	Let  Assumption \ref{ass3} hold with $F_1=F_2$. Then, for  $\boldsymbol{\lambda}\in\mathcal{D_+}$,  we have
	$$(n_{1},n_{2})\succeq_{w}(n^{*}_{1},n^{*}_{2})\Rightarrow X_{n^{*}:n^{*}}(n^{*}_{1},n^{*}_{2})\leq_{st}X_{n:n}(n_{1},n_{2}).$$
\end{corollary}

Next, we observe that two largest order statistics $X_{n:n}(n_{1},n_{2})$ and $Y_{n^{*}:n^{*}}(n^{*}_{1},n^{*}_{2})$ are comparable with respect to the usual stochastic order. It is worth mentioning that the order statistics are constructed from two multiple-outlier dependent samples having sample sizes $n$ and $ n^{*}.$  The pairs of the sizes of both the outliers $(n_1,n_2)$ and $(n^{*}_1,n^{*}_2)$ are assumed to be connected according to the weakly submajorization order. The following assumption is useful for the next theorem.
\begin{assumption}\label{ass4}
	Let $X_{1},\cdots,X_{n}$ be  $n$ nonnegative dependent random variables sharing Archimedean copula with generator $\psi_1,$ such that  $X_{i}\sim F_1(x\lambda_{1}),$ for $i=1,\cdots,n_1$ and $X_{j}\sim F_2(x\lambda_{2}),$ for $j={n_{1}+1},\cdots,n.$ Also, let $Y_{1},\cdots,Y_{n^*}$ be  $n^*$ dependent nonnegative random variables sharing Archimedean copula with generator $\psi_2,$ such that  $Y_{i}\sim F_1(x\mu_{1}),$ for $i=1,\cdots,n^*_1$ and $Y_{j}\sim F_2(x\mu_{2}),$ for $j={n^*_{1}+1},\cdots,n^*.$  Here, $1\leq n_{1}\leq n^{*}_{1}\leq n^{*}_{2}\leq n_{2}$, $n=n_1+n_2$ and $n^*=n^*_1+n^*_2.$ 
\end{assumption}

\begin{theorem}\label{th4}
	Assume that Assumption \ref{ass4} hold with $\tilde{r}_1(x)\leq\tilde{r}_2(x)$.
	 Let $(n_{1},n_{2})\succeq_{w}(n^{*}_{1},n^{*}_{2})$. Then,
	 $$(\underbrace{\lambda_{1},\cdots,\lambda_{1}}_{n^*_{1}},\underbrace{\lambda_{2},\cdots,\lambda_{2}}_{n^*_{2}})\succeq^{w}(\underbrace{\mu_{1},\cdots,\mu_{1}}_{n^{*}_{1}},\underbrace{\mu_{2},\cdots,\mu_{2}}_{n^{*}_{2}})\Rightarrow Y_{n^{*}:n^{*}}(n^{*}_{1},n^{*}_{2})\leq_{st}X_{n:n}(n_{1},n_{2}),$$ provided $\boldsymbol{ \lambda},~\boldsymbol{\mu}\in\mathcal{D}_+,$
	$\phi_2\circ\psi_1$ is super-additive,  $\psi_1\text{ or }\psi_2$ is log-convex and $\tilde{r}_1(x)\text{ or }\tilde{r}_2(x)$ is decreasing.
	
\end{theorem}
\begin{proof}
	By Theorem \ref{th1}, we have
	\begin{equation}\label{eq11}
		 (\underbrace{\lambda_{1},\cdots,\lambda_{1}}_{n^*_{1}},\underbrace{\lambda_{2},\cdots,\lambda_{2}}_{n^*_{2}})\succeq^{w}(\underbrace{\mu_{1},\cdots,\mu_{1}}_{n^{*}_{1}},\underbrace{\mu_{2},\cdots,\mu_{2}}_{n^{*}_{2}})\Rightarrow Y_{n^{*}:n^{*}}(n^{*}_{1},n^{*}_{2})\leq_{st}X_{n^{*}:n^{*}}(n^{*}_{1},n^{*}_{2}).
	\end{equation}
	  Further, from Theorem \ref{th3}, we have
	  \begin{equation}\label{eq12}
	  (n_{1},n_{2})\succeq_{w}(n^{*}_{1},n^{*}_{2})\Rightarrow X_{n^{*}:n^{*}}(n^{*}_{1},n^{*}_{2})\leq_{st}X_{n:n}(n_{1},n_{2}).
	  \end{equation}
	Upon combining inequalities given by \eqref{eq11} and \eqref{eq12}, the required result readily follows.
\end{proof}

The result stated in Theorem \ref{th4} is general. However, if we take some restrictions on the generators of the Archimedean copula and on the cumulative distribution functions, then we get some particular results. These are presented in the following corollary, which follows from Theorem \ref{th4}.
\begin{corollary}\label{cor4}
	Let the set-up in Assumption \ref{ass4} hold with $(n_{1},n_{2})\succeq_{w}(n^{*}_{1},n^{*}_{2}).$ Also, let $\boldsymbol{ \lambda},~\boldsymbol{\mu}\in\mathcal{D}_+,$ $\psi_1=\psi_2=\psi$ and $\psi$ be log-convex. Then,
	\begin{itemize}
		\item[(i)]
		$
		 (\underbrace{\lambda_{1},\cdots,\lambda_{1}}_{n^*_{1}},\underbrace{\lambda_{2},\cdots,\lambda_{2}}_{n^*_{2}})\succeq^{w}(\underbrace{\mu_{1},\cdots,\mu_{1}}_{n^{*}_{1}},\underbrace{\mu_{2},\cdots,\mu_{2}}_{n^{*}_{2}})\Rightarrow Y_{n^{*}:n^{*}}(n^{*}_{1},n^{*}_{2})\leq_{st}X_{n:n}(n_{1},n_{2}),$ provided
		$\tilde{r}_1(x)\text{ or }\tilde{r}_2(x)$ is decreasing and $\tilde{r}_1(x)\leq\tilde{r}_2(x)$.
		\item[(ii)]
		 $(\underbrace{\lambda_{1},\cdots,\lambda_{1}}_{n^*_{1}},\underbrace{\lambda_{2},\cdots,\lambda_{2}}_{n^*_{2}})\succeq^{w}(\underbrace{\mu_{1},\cdots,\mu_{1}}_{n^{*}_{1}},\underbrace{\mu_{2},\cdots,\mu_{2}}_{n^{*}_{2}})\Rightarrow Y_{n^{*}:n^{*}}(n^{*}_{1},n^{*}_{2})\leq_{st}X_{n:n}(n_{1},n_{2}),$ provided $\tilde{r}_1(x)=\tilde{r}_2(x)=\tilde{r}(x)$ is decreasing.
	\end{itemize}

\end{corollary}

We now present a numerical example, which provides an illustration of Theorem \ref{th4}.
\begin{example}\label{ex1}
Set $\boldsymbol{\lambda}=(\lambda_1,\lambda_2)=(5,2),~\boldsymbol{\mu}=(\mu_1,\mu_2)=(6,3),~(n_1,n_2)=(1,11),~(n^*_1,n^*_2)=(5,6),~\psi_1(x)=e^{-x^{\frac{1}{9}}},~\psi_2(x)=e^{-x^{\frac{1}{10}}},~x>0$. Consider the baseline distribution functions as $F_2(x)=1-e^{1-(1+x^{2})^{\frac{1}{5}}}$ and $F_1(x)=1-e^{-x},~x>0.$ Here, both the reversed hazard rate functions $\tilde{r}_1$ and $\tilde{r}_2$ are decreasing and satisfy $\tilde{r}_1(x)\leq\tilde{r}_2(x),$ for $x>0.$ Further, $\psi_1$ and $\psi_2$ are log-convex, $\phi_2\circ\psi_1$ is super-additive.
Thus, all the conditions of Theorem \ref{th4} are satisfied. Now,
we plot the graphs of ${F}_{X_{12:12}}(1,11)(x)$ and ${F}_{Y_{11:11}}(5,6)(x)$ in Figure $1a$, which shows that $Y_{11:11}(5,6)\leq_{st}X_{12:12}(1,11)$ holds.

\begin{figure}[h]
	\begin{center}
		\subfigure[]{\label{c1}\includegraphics[height=2.46in]{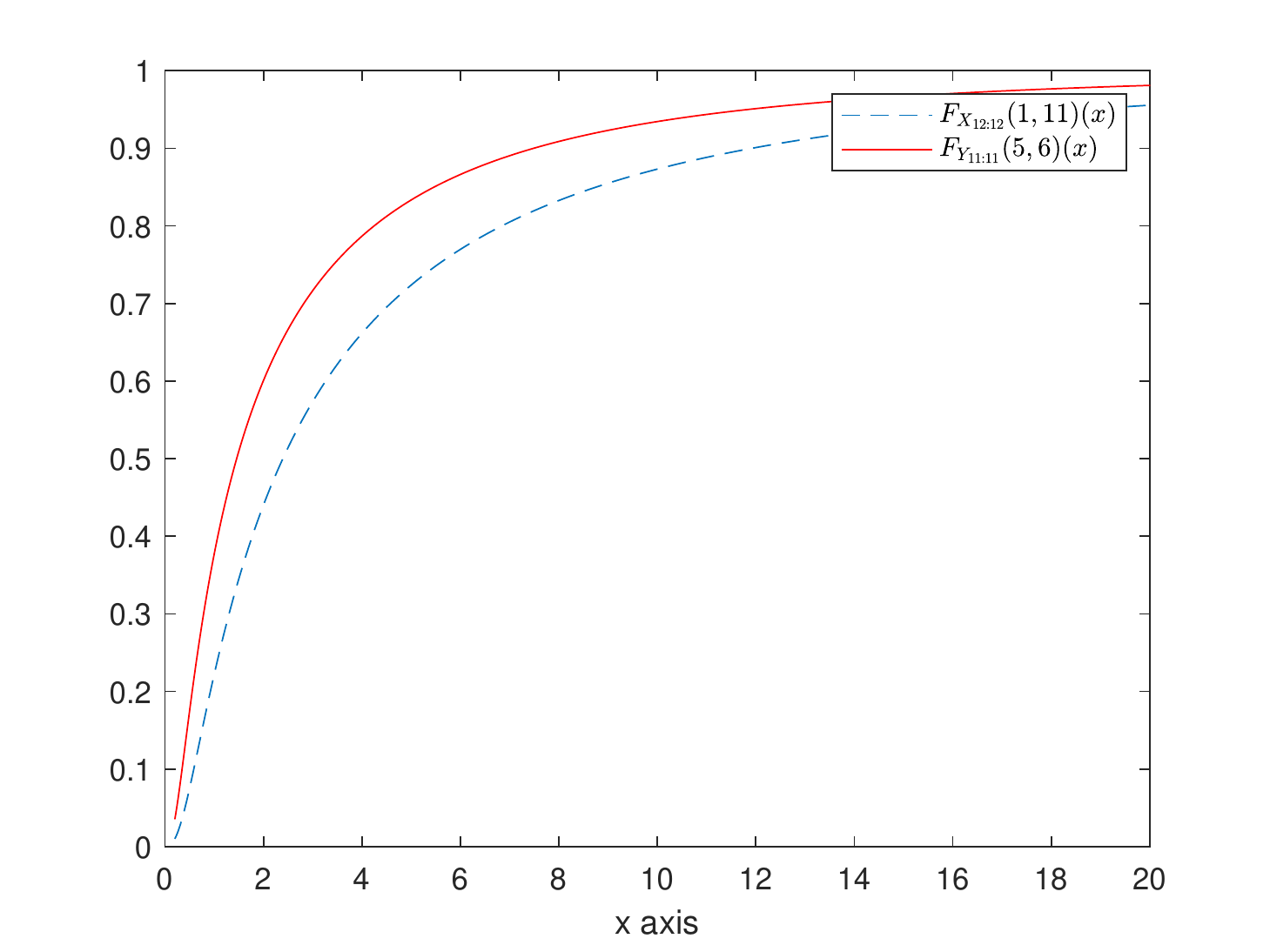}}
		\subfigure[]{\label{c2}\includegraphics[height=2.46in]{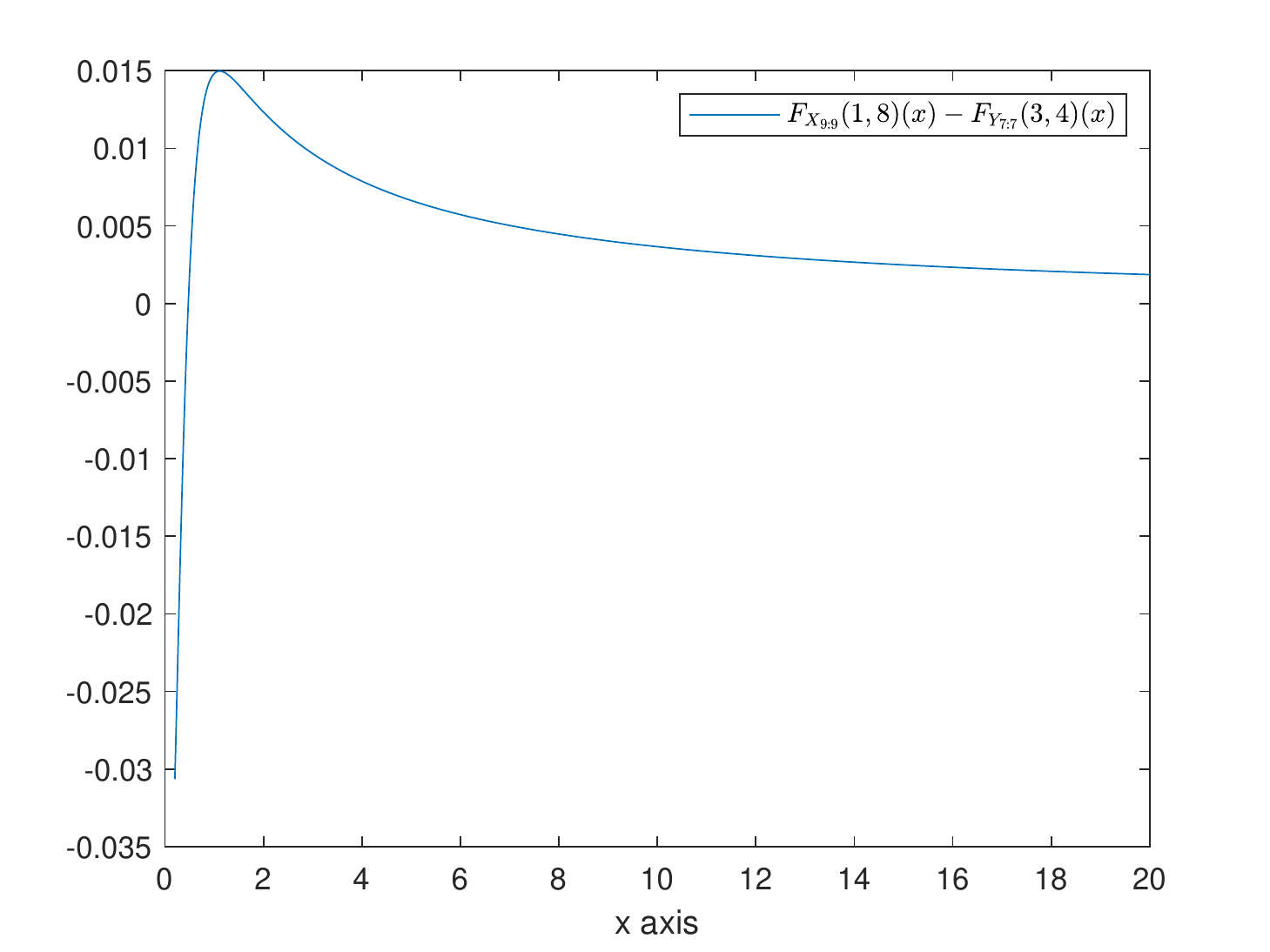}}
		\caption{
			(a) Plots of the distribution functions ${F}_{X_{12:12}}(1,11)(x)$ and ${F}_{Y_{11:11}}(5,6)(x)$ as in Example \ref{ex1}.
			 (b) Plot of $F_{X_{9:9}}(1,8)(x)-F_{Y_{7:7}}(3,4)(x)$ as in Counterexample \ref{ce1}.
			}
	\end{center}
\end{figure}
\end{example}
Next, we present a counterexample to illustrate that the result does not hold if  $\tilde{r}_1(x)\geq\tilde{r}_2(x)$ and $\boldsymbol{ \lambda}\in\mathcal{E_+}$ in Theorem \ref{th4}.
\begin{counterexample}\label{ce1}
	 Consider $\boldsymbol{\lambda}=(\lambda_1,\lambda_2)=(2,6),~\boldsymbol{\mu}=(\mu_1,\mu_2)=(8,2),~(n_1,n_2)=(1,8),~(n^*_1,n^*_2)=(3,4),~\psi_1(x)=e^{-x^{\frac{1}{3}}},~\psi_2(x)=e^{-x^{\frac{1}{10}}},~x>0.$ Baseline distribution functions are taken as $F_1(x)=1-e^{-x}$ and $F_2(x)=1-(1+2x)^{-0.5},~x>0.$
	It can be seen that all the conditions of Theorem \ref{th4} are satisfied except $\boldsymbol{ \lambda}\in\mathcal{D_+}$ and $\tilde{r}_1(x)\leq\tilde{r}_2(x)$. Now,
	we plot the graph of $F_{X_{9:9}}(1,8)(x)-F_{Y_{7:7}}(3,4)(x)$ in Figure $1b$, which reveals that $Y_{7:7}(3,4)\nleq_{st}X_{9:9}(1,8)$.
\end{counterexample}

%We now present a numerical example, which provides an illustration of Corollary \ref{cor4}(ii). Set $\boldsymbol{\lambda}=(2,5),~\boldsymbol{\mu}=(3,10),~(n_1,n_2)=(1,8),~(n^*_1,n^*_2)=(5,6),~\psi=e^{-x^{\frac{1}{1.5}}},~x>0.$ Also, consider the baseline distribution as $F(x)=1-exp(1-x^{0.5}),~x\geq1.$
%It is clear that, all the conditions of Corollary \ref{cor4}(ii) are satisfied. Now,
%we plot the graphs of ${F}_{X_{9:9}}(1,8)(x)$ and ${F}_{Y_{11:11}}(5,6)(x)$ in Figure $1a$, which gives $X_{9:9}(1,8)\leq_{st}Y_{11:11}(5,6)$ is satisfied.
%
%\begin{figure}[h]
%\begin{center}
%\subfigure[]{\label{c4}\includegraphics[height=2.41in]{M_Oth3_3.eps}}
%%\subfigure[]{\label{c12}\includegraphics[height=2.41in]{.eps}}
%\caption{
%	(a) Plots of ${F}_{X_{9:9}}(1,8)(x)$ and ${F}_{Y_{11:11}}(5,6)(x)$ .
%	}
%\end{center}
%\end{figure}
%Now a common question arises, under the same sufficient conditions as given in Theorem \ref{th4}, can we extend the usual stochastic order to the reversed hazard rate order. Figure $1b$ gives the answer. Here, we consider the same setup as given in Example \ref{ex1} and we will see that the ratio of  ${F}_{Y_{11:11}}(5,6)(x)$ and ${F}_{X_{9:9}}(1,8)(x)$ is not increasing. That implies, we are facing difficulty to compare two largest order statistics according to the reversed hazard rate order under the set up of Theorem \ref{th4}. Therefore, we need to change the sufficient conditions. Hence, we take the baseline distribution functions and the generators are  equal.\\

In the preceding theorems, we have derived sufficient conditions, under which the largest order statistics from multiple-outlier dependent scale models obey the usual stochastic order. However, naturally, it is of interest to extend the ordering results to some other stronger concepts of the stochastic orders.
In this part of the subsection, we  establish sufficient conditions, under which the reversed hazard rate order holds between the largest order statistics. The following theorem shows that the largest order statistics $X_{n^{*}:n^{*}}(n^{*}_{1},n^{*}_{2})$ and $Y_{n^{*}:n^{*}}(n^{*}_{1},n^{*}_{2})$ have the reversed hazard rate ordering when the scale parameters are associated with the weakly supermajorization order. The samples are heterogeneous and follow multiple-outlier dependent scale models.
\begin{theorem}\label{th7}
	Let Assumption \ref{ass1} hold with $r_1=r_2=r$, $n^{*}_{1}\geq(\leq)~n^{*}_2$ and $\psi_1=\psi_2=\psi$. Also, suppose $\psi$ is log-concave, $\frac{1-\psi}{\psi'}$ is decreasing and $\frac{1-\psi}{\psi'}[\frac{1-\psi}{\psi'}]'$ is increasing. Then,
	 $$(\underbrace{\lambda_{1},\cdots,\lambda_{1}}_{n^{*}_{1}},\underbrace{\lambda_{2},\cdots,\lambda_{2}}_{n^{*}_{2}})\succeq^{w}
	(\underbrace{\mu_{1},\cdots,\mu_{1}}_{n^{*}_{1}},\underbrace{\mu_{2},\cdots,\mu_{2}}_{n^{*}_{2}})\Rightarrow Y_{n^{*}:n^{*}}(n^{*}_{1},n^{*}_{2})\leq_{rh}X_{n^{*}:n^{*}}(n^{*}_{1},n^{*}_{2}),$$ provided $\boldsymbol{ \lambda},~\boldsymbol{\mu}\in\mathcal{E_+}~(\mathcal{D_+}),$
	 $r(x)$ is decreasing and $xr(x)$ is decreasing and convex.
\end{theorem}

\begin{proof}
	Under the given assumption $r_1=r_2=r$ implies $F_1=F_2=F.$
The reversed hazard rate function of $X_{n^{*}:n^{*}}(n^{*}_{1},n^{*}_{2})$ is
	\begin{eqnarray}\label{eq3.13}
	 \tilde{r}_{X_{n^{*}:n^{*}}}(n^{*}_{1},n^{*}_{2})(x)&=& \frac{\psi'\left[n^{*}_1\phi\left(F\left(x\lambda_1\right)\right)+n^{*}_2\phi\left(F\left(x\lambda_2\right)\right) \right]}{\psi\left[n^{*}_1\phi\left(F\left(x\lambda_1\right)\right)+n^{*}_2\phi\left(F\left(x\lambda_2\right)\right) \right]}\left[\frac{n^{*}_1\lambda_{1}f(x\lambda_1)}{\psi'[\phi\left(F\left(x\lambda_1\right)\right)]}+\frac{n^{*}_2\lambda_{2}f(x\lambda_2)}{\psi'[\phi\left(F\left(x\lambda_{2}\right)\right)]}\right]\nonumber\\
	 &=&\frac{\psi'\left[n^{*}_1\phi\left(F\left(x\lambda_1\right)\right)+n^{*}_2\phi\left(F\left(x\lambda_2\right)\right) \right]}{\psi\left[n^{*}_1\phi\left(F\left(x\lambda_1\right)\right)+n^{*}_2\phi\left(F\left(x\lambda_2\right)\right) \right]}\Bigg[\frac{n^{*}_1\lambda_{1}r(x\lambda_{1})[1-\psi[\phi\left(F\left(x\lambda_1\right)\right)]]}{\psi'[\phi\left(F\left(x\lambda_1\right)\right)]}\nonumber\\
	 &~&+\frac{n^{*}_2\lambda_{2}r(x\lambda_{2})[1-\psi[\phi\left(F\left(x\lambda_2\right)\right)]]}{\psi'[\phi\left(F\left(x\lambda_2\right)\right)]}\Bigg],
	\end{eqnarray}
	where $f$ is the probability density function corresponding to $F.$
		Denote $z=n^{*}_1\phi\left(F\left(x\lambda_1\right)\right)+n^{*}_2\phi\left(F\left(x\lambda_2\right)\right).$
	The partial derivative of $\tilde{r}_{X_{n^{*}:n^{*}}}(n^{*}_{1},n^{*}_{2})(x)$ with respect to $\lambda_1$ is obtained as
	\begin{eqnarray}\label{eq16}
	 \frac{\partial[\tilde{r}_{X_{n^{*}:n^{*}}}(n^{*}_{1},n^{*}_{2})(x)]}{\partial\lambda_1}&=&{n^{*}_1xr(x\lambda_{1})}\frac{d}{dz}\left[\frac{\psi'(z)}{\psi(z)}\right]\left[\frac{1-{\psi}\left[\phi\left[F\left(x\lambda_1\right)\right]\right]}{\psi'[\phi\left(F\left(x\lambda_1\right)\right)]}\right]\left[\sum_{i=1}^{n^{*}}\frac{\lambda_{i}f\left(x\lambda_i\right)}{\psi'[\phi\left(F\left(x\lambda_i\right)\right)]}\right]
	\nonumber\\
	 &~&+n^{*}_1x\lambda_{1}\left[r(x\lambda_{1})\right]^2\frac{\psi'(z)}{\psi(z)}\left[\frac{1-{\psi}(v)}{{\psi}'(v)}\frac{d}{dv}\left[\frac{1-{\psi}(v)}{{\psi}'(v)}\right]\right]_{v=\phi(F(x\lambda_{1}))}\nonumber\\
	 &~&+n^{*}_1\frac{d}{dw}\left[wr(w)\right]_{w=x\lambda_{1}}\left[\frac{1-{\psi}\left[\phi\left[F\left(x\lambda_1\right)\right]\right]}{{\psi}'\left[\phi\left[F\left(x\lambda_1\right)\right]\right]}\right]\frac{\psi'(z)}{\psi(z)}.
	\end{eqnarray}
	Utilizing Theorem $A.8$ of \cite{Marshall2011}, to obtain the desired result, we need to prove that $\tilde{r}_{X_{n^{*}:n^{*}}}(n^{*}_{1},n^{*}_{2})(x)$ is decreasing and Schur-convex with respect to $\boldsymbol{ \lambda}.$ Using the given assumptions and Equation \eqref{eq16}, the decreasing property of $\tilde{r}_{X_{n^{*}:n^{*}}}(n^{*}_{1},n^{*}_{2})(x)$ with respect to $\boldsymbol{ \lambda}$ is obvious. Further,
  according to Lemma \ref{lem2.1b} (Lemma \ref{lem2.1a}), to show Schur-convexity of $\tilde{r}_{X_{n^{*}:n^{*}}}(n^{*}_{1},n^{*}_{2})(x),$  we have to establish that
   for  $1\leq i\leq j\leq n^{*},$
 \begin{equation}\label{eq3.3}
 \left[\frac{\partial [\tilde{r}_{X_{n^{*}:n^{*}}}(n^{*}_{1},n^{*}_{2})(x)]}{\partial\lambda_i}-\frac{\partial [\tilde{r}_{X_{n^{*}:n^{*}}}(n^{*}_{1},n^{*}_{2})(x)]}{\partial\lambda_j}\right]\leq(\geq) 0, ~\text{for} ~\boldsymbol{ \lambda}\in\mathcal{E}_+~(\mathcal{D}_+).
 \end{equation}
	Now, consider the following three cases.\\
	Case I: For $1\leq i\leq j\leq n^{*}_1 ,$ $\lambda_i=\lambda_j=\lambda_1.$  Here, $\frac{\partial [\tilde{r}_{X_{n^{*}:n^{*}}}(n^{*}_{1},n^{*}_{2})(x)]}{\partial\lambda_i}-\frac{\partial [\tilde{r}_{X_{n^{*}:n^{*}}}(n^{*}_{1},n^{*}_{2})(x)]}{\partial\lambda_j}=0$.\\
		Case II: For $n^{*}_{1}+1\leq i\leq j\leq n^{*} ,$ $\lambda_i=\lambda_j=\lambda_2.$  Hence, $\frac{\partial [\tilde{r}_{X_{n^{*}:n^{*}}}(n^{*}_{1},n^{*}_{2})(x)]}{\partial\lambda_i}-\frac{\partial [\tilde{r}_{X_{n^{*}:n^{*}}}(n^{*}_{1},n^{*}_{2})(x)]}{\partial\lambda_j}=0$.\\
	Case III: For $1\leq i\leq n^{*}_1 $ and $n^{*}_{1}+1\leq j\leq n^{*},$ $\lambda_i=\lambda_1$ and $\lambda_j=\lambda_2$. Consider $\lambda_1\leq\lambda_2$, which gives $\phi(F(x\lambda_1))\geq\phi(F(x\lambda_2)).$ Here, we only discuss the proof when $\lambda_1\leq\lambda_2.$ The other case when $\lambda_1\ge\lambda_2$ can be proved in the similar way. The concavity property of  $\ln\psi$ provides $\frac{d}{dz}\left[\frac{\psi'(z)}{\psi(z)}\right]\leq 0.$ Again, using decreasing property of  $\frac{1-\psi}{{\psi}'}$, we have
	\begin{equation}\label{eq17}
	\frac{1-\psi(w)}{{\psi}'(w)}\Big|_{w=\phi[F(x\lambda_{1})]}\leq \frac{1-\psi(w)}{{\psi}'(w)}\Big|_{w=\phi[F(x\lambda_{2})]}\leq 0.
	\end{equation}
	Further, it has been assumed that $r(x)$ is decreasing, $x r(x)$ is decreasing and convex. Therefore, using $n^{*}_{1}\geq n^{*}_{2},$ we have
	\begin{equation}\label{eq18.}
	 r(x\lambda_{1})\geq  r(x\lambda_{2}),
	\end{equation}
	\begin{equation}\label{eq18}
	n^{*}_{1}x\lambda_{1} r(x\lambda_{1})\geq n^{*}_{2}x\lambda_{2} r(x\lambda_{2}) \text{ ~~~~    and }
	\end{equation}
	\begin{equation}\label{eq19}
n^{*}_{1}	\frac{d}{dw}\left[wr(w)\right]_{w=x\lambda_{1}}\leq n^{*}_{2}\frac{d}{dw}\left[wr(w)\right]_{w=x\lambda_{2}}\leq 0.
	\end{equation}
	Moreover, $\frac{1-\psi(w)}{{\psi}'(w)}\frac{d}{dw}[\frac{1-\psi(w)}{{\psi}'(w)}]$ is increasing. Therefore, we obtain the following inequality
	\begin{equation}\label{eq20}
	 \left[\frac{1-\psi(w)}{{\psi}'(w)}\frac{d}{dw}\left[\frac{1-\psi(w)}{{\psi}'(w)}\right]\right]_{w=\phi[F(x\lambda_{1})]}\geq \left[\frac{1-\psi(w)}{{\psi}'(w)}\frac{d}{dw}\left[\frac{1-\psi(w)}{{\psi}'(w)}\right]\right]_{w=\phi[F(x\lambda_{2})]}\geq 0.
	\end{equation}
	Now, combining \eqref{eq17}-\eqref{eq20} and the given assumptions, we obtain that the inequality given by \eqref{eq3.3} holds. This completes the proof.
	\end{proof}
%The following corollary is a direct consequence of Theorem \ref{th7}.
%\begin{corollary}\label{cor6}
%	Let Set-up \ref{ass1} hold with $\psi_1=\psi_2=\psi,$ $F_{1}(x)=F_{2}(x)$ and $r_1(x)=r_2(x)=r(x)$. Also, suppose $\boldsymbol{ \lambda},~\boldsymbol{\mu}\in\mathcal{E_+}(~\mathcal{D_+}),$
%	$\psi$ is log-concave, $\frac{1-\psi}{\psi'}$ is decreasing, $\frac{1-\psi}{\psi'}[\frac{1-\psi}{\psi'}]'$ is increasing.
% Then,
%		 $$(\underbrace{\lambda_{1},\cdots,\lambda_{1}}_{n_{1}},\underbrace{\lambda_{2},\cdots,\lambda_{2}}_{n_{2}})\succeq^{w}
%		(\underbrace{\mu_{1},\cdots,\mu_{1}}_{n_{1}},\underbrace{\mu_{2},\cdots,\mu_{2}}_{n_{2}})\Rightarrow Y_{n:n}(n_{1},n_{2})\leq_{rh}X_{n:n}(n_{1},n_{2}),$$ provided  $r(x)$ is decreasing, $xr(x)$ is decreasing and convex.
%\end{corollary}

In the next theorem, we show that the largest order statistics $X_{n:n}(n_{1},n_{2})$ and $X_{n^{*}:n^{*}}(n^{*}_{1},n^{*}_{2})$ are comparable according to the reversed hazard rate order.
\begin{theorem}\label{th8}
Let Assumption \ref{ass3} hold with $\psi_1=\psi$ and $r_1= r_2=r.$ Then, for  $\boldsymbol{\lambda}\in\mathcal{D_+},$ we have
$$(n_{1},n_{2})\succeq_{w}(n^{*}_{1},n^{*}_{2})\Rightarrow X_{n^{*}:n^{*}}(n^{*}_{1},n^{*}_{2})\leq_{rh}X_{n:n}(n_{1},n_{2}),$$ provided $\ln\psi$ is concave, $\frac{1-\psi}{\psi'}$ and $xr(x)$ are decreasing.
\end{theorem}
\begin{proof}
The stated result will be proved, if we show that $\tilde{r}_{X_{n:n}}(n_{1},n_{2})(x)\geq 	\tilde{r}_{X_{n^{*}:n^{*}}}(n^{*}_{1},n^{*}_{2})(x).$ Equivalently,
\begin{equation}\label{eq22}
 \frac{\psi'\left[\displaystyle\sum_{i=1}^{n}\phi\left(F\left(x\lambda_i\right)\right)\right]}{\psi\left[\displaystyle\sum_{i=1}^{n}\phi\left(F\left(x\lambda_i\right)\right)\right]}\times\frac{\psi\left[\displaystyle\sum_{i=1}^{n^{*}}\phi\left(F\left(x\lambda_i\right)\right)\right]}{\psi'\left[\displaystyle\sum_{i=1}^{n^{*}}\phi\left(F\left(x\lambda_i\right)\right)\right]}
 \geq\displaystyle\frac{ \displaystyle\sum_{i=1}^{n^{*}}\frac{\lambda_{i}r(x\lambda_{i})[1-\psi[\phi\left(F\left(x\lambda_i\right)\right)]]}{\psi'[\phi\left(F\left(x\lambda_i\right)\right)]}}{\displaystyle\sum_{i=1}^{n}\frac{\lambda_{i}r(x\lambda_{i})[1-\psi[\phi\left(F\left(x\lambda_i\right)\right)]]}{\psi'[\phi\left(F\left(x\lambda_i\right)\right)]}}.
\end{equation}
The preceding inequality holds if the following two inequalities are satisfied,
\begin{equation}\label{eq23}
	 \frac{\psi'\left[\sum_{i=1}^{n^*}\phi\left(F\left(x\lambda_i\right)\right)\right]}{\psi\left[\sum_{i=1}^{n^*}\phi\left(F\left(x\lambda_i\right)\right)\right]}\geq \frac{\psi'\left[\sum_{i=1}^{n}\phi\left(F\left(x\lambda_i\right)\right)\right]}{\psi\left[\sum_{i=1}^{n}\phi\left(F\left(x\lambda_i\right)\right)\right]}
	\Leftrightarrow (n^{*}_{1}-n_{1})\phi(F(x\lambda_{1}))\leq  (n_{2}-n^{*}_{2})\phi(F(x\lambda_{2}))
\end{equation}
and
\begin{eqnarray}\label{eq24}
	 \sum_{i=1}^{n^*}\frac{\lambda_{i}r(x\lambda_{i})[1-\psi[\phi\left(F\left(x\lambda_i\right)\right)]]}{\psi'[\phi\left(F\left(x\lambda_i\right)\right)]}&\geq&\sum_{i=1}^{n}\frac{\lambda_{i}r(x\lambda_{i})[1-\psi[\phi\left(F\left(x\lambda_i\right)\right)]]}{\psi'[\phi\left(F\left(x\lambda_i\right)\right)]}\nonumber\\\Leftrightarrow
	 (n^{*}_1-n_{1})\frac{\lambda_{1}r(x\lambda_{1})[1-\psi[\phi\left(F\left(x\lambda_1\right)\right)]]}{\psi'[\phi\left(F\left(x\lambda_1\right)\right)]}&\geq&(n_{2}-n^{*}_{2})\frac{\lambda_{2}r(x\lambda_{2})[1-\psi[\phi\left(F\left(x\lambda_2\right)\right)]]}{\psi'[\phi\left(F\left(x\lambda_2\right)\right)]}.
	\end{eqnarray}
	Furthermore,
	\begin{eqnarray}\label{eq25}
	(n_{1},n_{2})\succeq_{w}(n^{*}_{1},n^{*}_{2})
	\Rightarrow(n_{1}+n_{2})\geq(n^{*}_{1}+n^{*}_{2})
	\Rightarrow(n_{2}-n^{*}_{2})\geq (n^{*}_{1}-n_{1})\geq 0.
	\end{eqnarray}
	 Also, $$\lambda_{1}\geq \lambda_2\Rightarrow\phi\left(F\left(x\lambda_2\right)\right)\geq\phi\left(F\left(x\lambda_1\right)\right)\geq0.$$
	Moreover, $\frac{1-\psi}{{\psi}'}$ is decreasing. Thus,
	\begin{equation}\label{eq26}
	\frac{1-\psi(w)}{{\psi}'(w)}\Big|_{w=\phi[F(x\lambda_{2})]}\leq \frac{1-\psi(w)}{{\psi}'(w)}\Big|_{w=\phi[F(x\lambda_{1})]}\leq 0.
	\end{equation}
	Using the decreasing property of $x r(x)$, we have
	\begin{equation}\label{eq27}
	x\lambda_{1} r(x\lambda_{1})\leq x\lambda_{2} r(x\lambda_{2}).
	\end{equation}
Combining Equations \eqref{eq25}, \eqref{eq26} and \eqref{eq27}, the inequality in \eqref{eq24} can be obtained. Using \eqref{eq25} and the assumption that $\psi$ is log-concave, we get the  inequality \eqref{eq23}. Hence, the proof follows.
\end{proof}
%From Theorem \ref{th8}, we have the following corollary.
%\begin{corollary}\label{cor8}
%	Let Set-up \ref{ass3} hold with $\psi_1=\psi_2=\psi$ and $r_1=r_2=r$. Then, for  $\boldsymbol{\mu}\in\mathcal{E_+}$ and $F_{1}=F_{2},$ we have
%	$$(n_{1},n_{2})\succeq^{w}(n^{*}_{1},n^{*}_{2})\Rightarrow Y_{n:n}(n_{1},n_{2})\leq_{rh}Y_{n^{*}:n^{*}}(n^{*}_{1},n^{*}_{2}),$$ provided $\psi$ is log-concave, $\frac{1-\psi}{\psi'}$ is decreasing and $xr(x)$ is decreasing.
%\end{corollary}
Now, we are ready to state a result which shows that the largest order statistics $X_{n:n}(n_{1},n_{2})$ and $Y_{n^{*}:n^{*}}(n^{*}_{1},n^{*}_{2})$ can be compared with respect to the reversed hazard rate order.
\begin{theorem}\label{th9}
Let the set-up in Assumption \ref{ass4} hold with $\psi_1=\psi_2=\psi$ and $r_1= r_2=r$. Also, assume $\boldsymbol{\lambda},~\boldsymbol{\mu}\in\mathcal{D_+}$ and
$(n_{1},n_{2})\succeq_{w}(n^{*}_{1},n^{*}_{2})$. Then,
$$(\underbrace{\lambda_{1},\cdots,\lambda_{1}}_{n^*_{1}},\underbrace{\lambda_{2},\cdots,\lambda_{2}}_{n^*_{2}})\succeq^{w}
(\underbrace{\mu_{1},\cdots,\mu_{1}}_{n^{*}_{1}},\underbrace{\mu_{2},\cdots,\mu_{2}}_{n^{*}_{2}})\Rightarrow  Y_{n^{*}:n^{*}}(n^{*}_{1},n^{*}_{2})\leq_{rh}X_{n:n}(n_{1},n_{2}),$$ provided $\psi$ is log-concave, $\frac{1-\psi}{\psi'}$ is decreasing, $\frac{1-\psi}{\psi'}[\frac{1-\psi}{\psi'}]'$ is increasing,  $xr(x)$ is decreasing, convex  and $r(x)$  is decreasing.
\end{theorem}
\begin{proof}
According to Theorem \ref{th7}, we have
\begin{equation}\label{eq28}
(\underbrace{\lambda_{1},\cdots,\lambda_{1}}_{n^*_{1}},\underbrace{\lambda_{2},\cdots,\lambda_{2}}_{n^*_{2}})\succeq^{w}
(\underbrace{\mu_{1},\cdots,\mu_{1}}_{n^{*}_{1}},\underbrace{\mu_{2},\cdots,\mu_{2}}_{n^{*}_{2}})\Rightarrow Y_{n^{*}:n^{*}}(n^{*}_{1},n^{*}_{2}) \leq_{rh}X_{n^{*}:n^{*}}(n^{*}_{1},n^{*}_{2}).
\end{equation}
Also, from Theorem \ref{th8}, we get
\begin{equation}\label{eq29}
(n_{1},n_{2})\succeq_{w}(n^{*}_{1},n^{*}_{2})\Rightarrow X_{n^{*}:n^{*}}(n^{*}_{1},n^{*}_{2})\leq_{rh}X_{n:n}(n_{1},n_{2}).
\end{equation}
Thus, the proof of the theorem follows after combining the inequalities given by \eqref{eq28} and \eqref{eq29}.
\end{proof}
%Applying Theorem \ref{th8}, we have the next corollary.
%\begin{corollary}\label{cor9}
%	Let Set-up \ref{ass4} hold with $r_1=r_2=r,~f=F_2$ and $\psi_1=\psi_2=\psi$. Also, assume  $\boldsymbol{\lambda},~\boldsymbol{\mu}\in\mathcal{E_+}$ and
%	$(n_{1},n_{2})\succeq^{w}(n^{*}_{1},n^{*}_{2})$. Then,
%	 $$(\underbrace{\mu_{1},\cdots,\mu_{1}}_{n_{1}},\underbrace{\mu_{2},\cdots,\mu_{2}}_{n_{2}})\succeq^{w}(\underbrace{\lambda_{1},\cdots,\lambda_{1}}_{n^*_{1}},\underbrace{\lambda_{2},\cdots,\lambda_{2}}_{n^*_{2}})\Rightarrow X_{n:n}(n_{1},n_{2}) \leq_{rh}Y_{n^{*}:n^{*}}(n^{*}_{1},n^{*}_{2}),$$ provided $\psi$ is log-concave, $\frac{1-\psi}{\psi'}$ is decreasing, $\frac{1-\psi}{\psi'}[\frac{1-\psi}{\psi'}]'$ is increasing,  $xr(x)$ is decreasing, convex and $r(x)$  is decreasing.
%\end{corollary}

Below, we consider an example to illustrate Theorem \ref{th9}.

\begin{example}\label{ex3.2}
 Consider $\boldsymbol{\lambda}=(3,2),~\boldsymbol{\mu}=(6,5),~(n_1,n_2)=(2,10),~(n^*_1,n^*_2)=(3,4),~\psi=e^{\frac{1}{0.2}(1-e^x)},~x>0.$ Also, let the baseline distribution be $F(x)=1-\left(\frac{x}{b}\right)^{-a},~x\geq b>0,~a>0.$
 It is not hard to see that for $a=5$ and $b=1,$ all the conditions of Theorem \ref{th9} are satisfied. Further, we plot the graph of $F_{X_{12:12}}(2,10)(x)/F_{Y_{7:7}}(3,4)(x)$  in Figure $2a$. This shows that the result in Theorem \ref{th9} holds.
\begin{figure}[h]
	\begin{center}
		\subfigure[]{\label{c3}\includegraphics[height=2.41in]{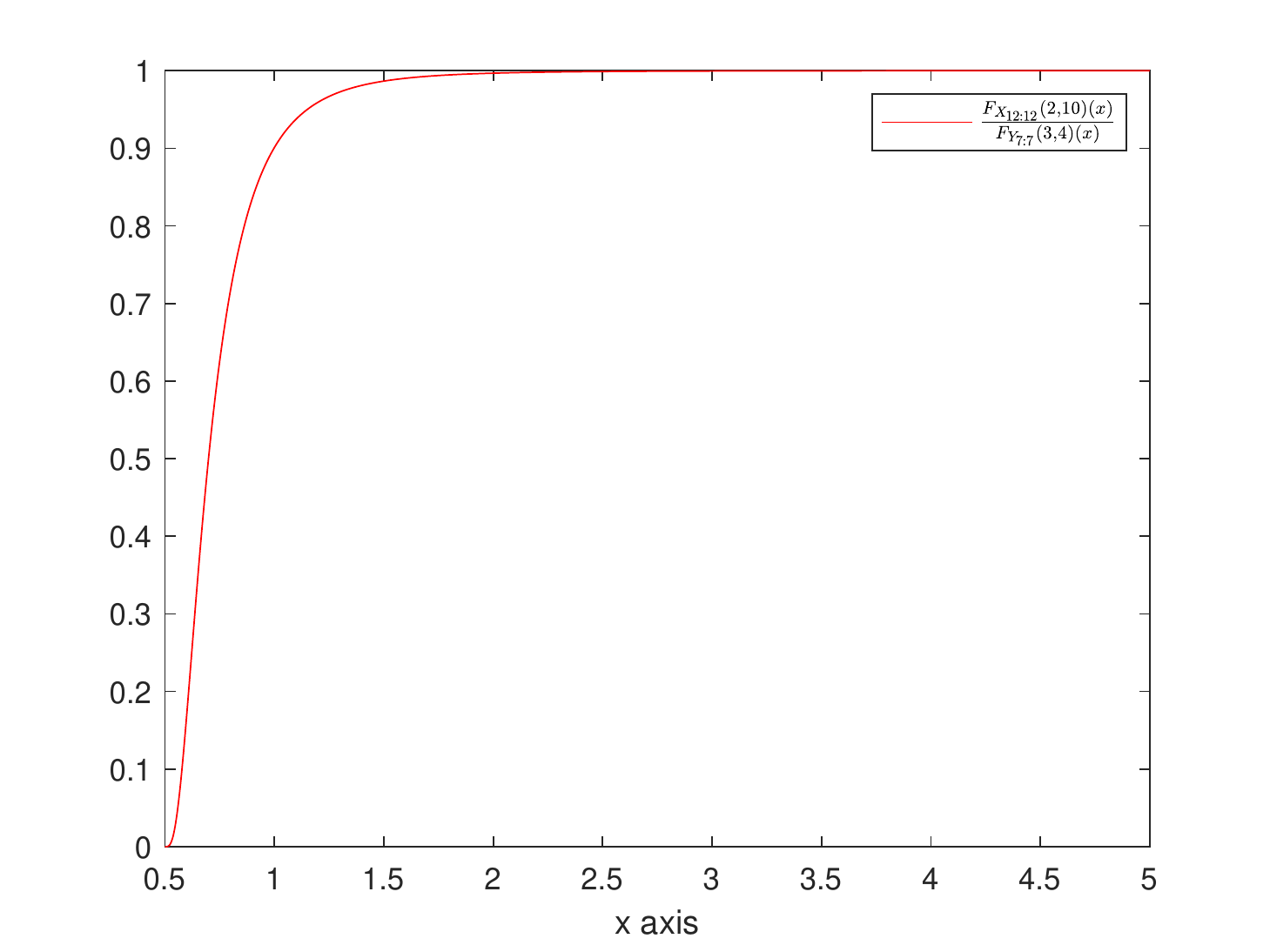}}
			\subfigure[]{\includegraphics[height=2.41in]{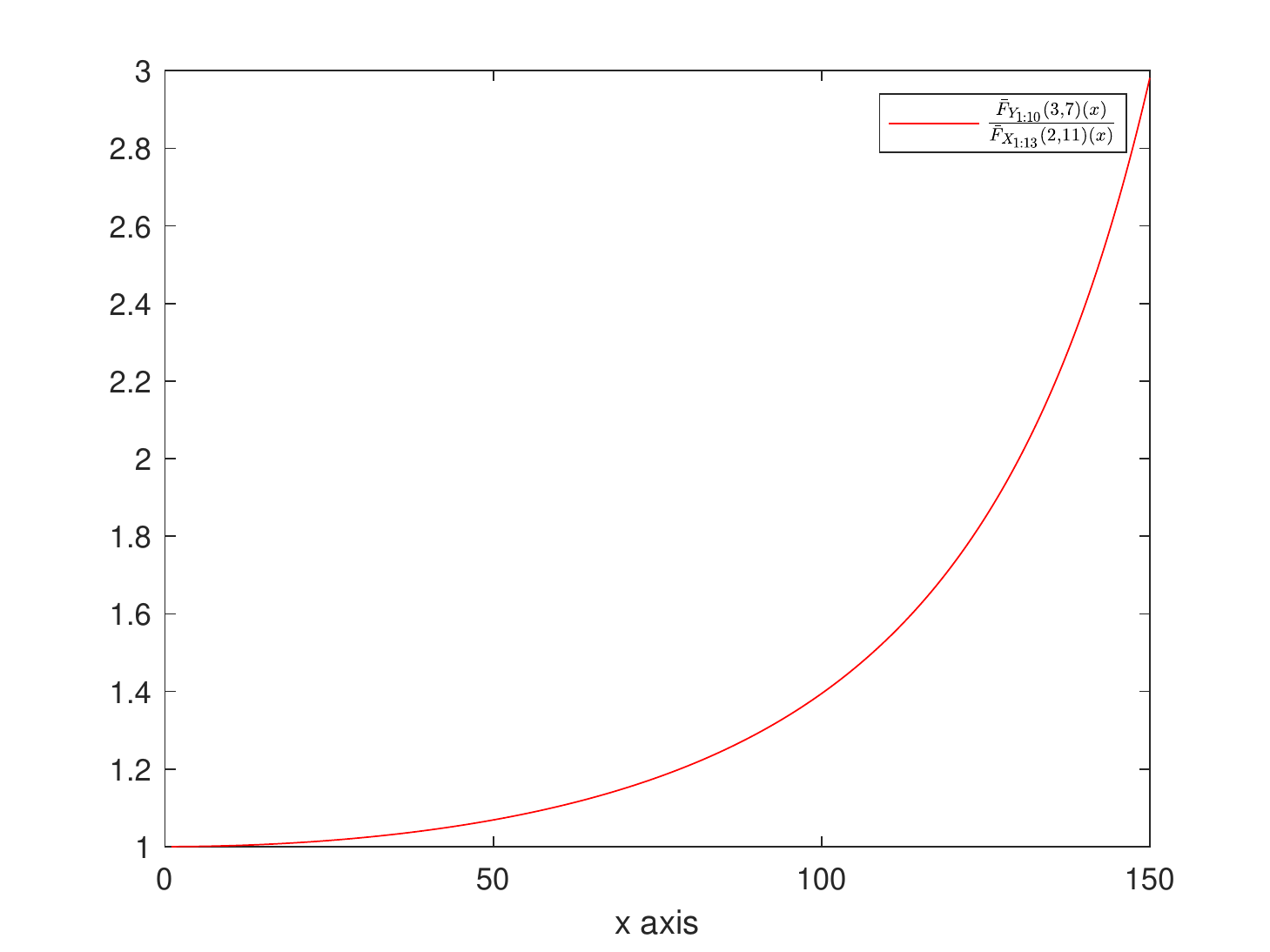}}
%			\subfigure[]{\includegraphics[height=2.41in]{M_Oth3_7_star.eps}}
		%			\subfigure[]{\label{c1.}\includegraphics[height=2.41in]{M_Oth3_3e.eps}}
		\caption{
			(a) Plot of the ratio of two distribution functions $F_{X_{12:12}}(2,10)(x)/F_{Y_{7:7}}(3,4)(x)$ in Example \ref{ex3.2}. (b) Plot of $\bar{F}_{Y_{1:10}}(3,7)(x)/\bar{F}_{X_{1:13}}(2,11)(x)$ as in Example \ref{ex3.5}.
			%(b) Plot of $\frac{{F}_{Y_{11:11}}(5,6)(x)}{{F}_{X_{9:9}}(1,8)(x)}$.
		}
	\end{center}
\end{figure}
\end{example}
Now, we derive conditions such that the star order holds between $X_{n^{*}:n^{*}}(n^{*}_{1},n^{*}_{2})$ and $Y_{n^{*}:n^{*}}(n^{*}_{1},n^{*}_{2})$. Denote $\lambda_{2:2}=\max\{\lambda_1,\lambda_2\}, ~\lambda_{1:2}=\min\{\lambda_1,\lambda_2\},~\mu_{2:2}=\max\{\mu_1,\mu_2\}$ and $\mu_{1:2}=\min\{\mu_1,\mu_2\}.$

\begin{theorem}\label{th}
	Under the set-up as in Assumption \ref{ass1}, with $\tilde{r}_1=\tilde{r}_2=\tilde{r}$ and $\psi_1=\psi_2=\psi$, we have
	$$\frac{\lambda_{2:2}}{\lambda_{1:2}}\geq\frac{\mu_{2:2}}{\mu_{1:2}}\Rightarrow Y_{n^{*}:n^{*}}(n^{*}_{1},n^{*}_{2})\leq_{*}X_{n^{*}:n^{*}}(n^{*}_{1},n^{*}_{2}),$$ provided
	$\frac{\psi}{\psi'}$ is decreasing, convex, $\frac{x\tilde{r}'(x)}{\tilde{r}(x)}$ is decreasing and  $x\tilde{r}(x)$ is increasing.
\end{theorem}

\begin{proof}
	Under the assumption, $\tilde{r}_1=\tilde{r}_2=\tilde{r}$ gives $F_1=F_2=F.$
	The distribution functions of $X_{n^{*}:n^{*}}(n^{*}_1,n^{*}_2)$ and $Y_{n^{*}:n^{*}}(n^{*}_1,n^{*}_2)$ are respectively given by
	 $$F_{X_{n^{*}:n^{*}}}(n^{*}_1,n^{*}_2)(x)=\psi\left[n^{*}_1\phi\left(F\left(x\lambda_1\right)\right)+n^{*}_2\phi\left(F\left(x\lambda_2\right)\right) \right]$$ and
	 $$F_{Y_{n^{*}:n^{*}}}(n^{*}_1,n^{*}_2)(x)=\psi\left[n^{*}_1\phi\left(F\left(x\mu_1\right)\right)+n^{*}_2\phi\left(F\left(x\mu_2\right)\right) \right].$$
	
	To obtain the required result, we consider two cases.\\
	Case-I: $\lambda_{1}+\lambda_{2}=\mu_{1}+\mu_{2}.$\\
	For convenience, we assume that $\lambda_{1}+\lambda_{2}=\mu_{1}+\mu_{2}=1.$ For this case, it is clear that $(\lambda_1,\lambda_2)\succeq^{m}(\mu_1,\mu_2)$. Now, take $\lambda_2=\lambda\geq \lambda_1,~\mu_2=\mu\geq\mu_1.$ Hence, $\lambda_{1}=1-\lambda$ and $\mu_{1}=1-\mu.$ Based on this, the distribution functions of $X_{n^{*}:n^{*}}(n^{*}_1,n^{*}_2)$ and $Y_{n^{*}:n^{*}}(n^{*}_1,n^{*}_2)$ can be written in the following form
	 $$F_{X_{n^{*}:n^{*}}}(n^{*}_1,n^{*}_2)(x)\overset{def}{=}F_{\lambda}(x)=\psi\left[n^{*}_1\phi\left(F\left(x(1-\lambda)\right)\right)+n^{*}_2\phi\left(F\left(x\lambda\right)\right) \right]$$ and
	 $$F_{Y_{n^{*}:n^{*}}}(n^{*}_1,n^{*}_2)(x)\overset{def}{=}F_{\mu}(x)=\psi\left[n^{*}_1\phi\left(F\left(x(1-\mu\right)\right)+n^{*}_2\phi\left(F\left(x\mu\right)\right) \right].$$
	Now, according to Lemma \ref{lem4}, we have to show that $\frac{F'_{\lambda}(x)}{xf_{\lambda}(x)}$ is decreasing in $x\in \mathbb{ R}^+,$ for $\lambda\in(1/2,1].$ The derivative of $F_\lambda$, with respect to $\lambda$ is given by
	\begin{eqnarray}\nonumber
		F'_{\lambda}(x)&=&
		 \left[-xn^{*}_1\tilde{r}(x(1-\lambda))\frac{\psi[\phi\left(F\left(x(1-\lambda)\right)\right)]}{\psi'[{\phi\left(F\left(x(1-\lambda)\right)\right)}]}+xn^{*}_2\tilde{r}(x\lambda)\frac{\psi[\phi\left(F\left(x\lambda\right)\right) ]}{\psi'[\phi\left(F\left(x\lambda\right)\right) ]}\right]\nonumber\\
		 &~&\times\psi'\left[n^{*}_1\phi\left(F\left(x(1-\lambda)\right)\right)+n^{*}_2\phi\left(F\left(x\lambda\right)\right) \right].
	\end{eqnarray}
	Also, the probability density function corresponding to $F_{\lambda}$ is
	\begin{eqnarray}
		f_{\lambda}(x)&=&
		\left[(1-\lambda) n^{*}_1\tilde{r}(x(1-\lambda))\frac{\psi[\phi\left(F\left(x(1-\lambda)\right)\right)]}{\psi'[{\phi\left(F\left(x(1-\lambda)\right)\right)}]}+\lambda n^{*}_2\tilde{r}(x\lambda)\frac{\psi[\phi\left(F\left(x\lambda\right)\right) ]}{\psi'[\phi\left(F\left(x\lambda\right)\right) ]}\right]\nonumber\\
		 &~&\times\psi'\left[n^{*}_1\phi\left(F\left(x(1-\lambda)\right)\right)+n^{*}_2\phi\left(F\left(x\lambda\right)\right) \right].
	\end{eqnarray}
	Therefore,
	\begin{eqnarray}\nonumber
		\frac{F'_{\lambda}(x)}{xf_{\lambda}(x)}=
		 \left(\lambda+\left[\frac{n^{*}_2\tilde{r}(x\lambda)\displaystyle\frac{\psi[\phi\left(F\left(x\lambda\right)\right) ]}{\psi'[\phi\left(F\left(x\lambda\right)\right) ]}}{n^{*}_1\tilde{r}(x(1-\lambda))\displaystyle\frac{\psi[\phi\left(F\left(x(1-\lambda)\right)\right) ]}{\psi'[\phi\left(F\left(x(1-\lambda)\right)\right) ]}}-1\right]^{-1}\right)^{-1}.
	\end{eqnarray}
	Thus, it suffices to show that
	$L(x)=\left(\tilde{r}(x\lambda)\frac{\psi[\phi\left(F\left(x\lambda\right)\right) ]}{\psi'[\phi\left(F\left(x\lambda\right)\right) ]}\right)/\left(\tilde{r}(x(1-\lambda))\frac{\psi[\phi\left(F\left(x(1-\lambda)\right)\right) ]}{\psi'[\phi\left(F\left(x(1-\lambda)\right)\right) ]}\right)$ is decreasing in $x\in\mathbb{ R}^+$, for $\lambda\in(1/2,1].$
	The derivative of $L(x)$ with respect to $x$ is obtained as
	\begin{eqnarray}
		 L'(x)&\overset{sign}{=}&\frac{\lambda\tilde{r}'(x\lambda)}{\tilde{r}(x\lambda)}+\lambda\tilde{r}(x\lambda)\left[\frac{\psi[\phi(F(x\lambda))]}{\psi'[\phi(F(x\lambda))]}\right]'-\frac{(1-\lambda)\tilde{r}'(x(1-\lambda))}{\tilde{r}((1-\lambda)x)}\nonumber\\
	 &~&-(1-\lambda)\tilde{r}(x(1-\lambda))\left[\frac{\psi[\phi(F(x(1-\lambda)))]}{\psi'[\phi(F(x(1-\lambda)))]}\right]'.\nonumber
	\end{eqnarray}
	Under the assumptions made, $\frac{x\tilde{r}'(x)}{\tilde{r}(x)}$ is decreasing and  $x\tilde{r}(x)$ is increasing. Therefore, for $\lambda\in(1/2,1],$
	\begin{equation}\label{s1}
		 \frac{x\lambda\tilde{r}'(x\lambda)}{\tilde{r}(x\lambda)}\leq\frac{x(1-\lambda)\tilde{r}'(x(1-\lambda))}{\tilde{r}(x(1-\lambda))}\leq 0\text{ and }x\lambda\tilde{r}(x\lambda) \geq x(1-\lambda)\tilde{r}(x(1-\lambda))\geq 0.
	\end{equation}
	Also, since $\frac{\psi}{\psi'}$ is decreasing and convex, we have
	\begin{equation}\label{s2}
		\left[\frac{\psi[\phi(F(x\lambda))]}{\psi'[\phi(F(x\lambda))]}\right]'\leq \left[\frac{\psi[\phi(F(x(1-\lambda)))]}{\psi'[\phi(F(x(1-\lambda)))]}\right]'\leq 0.
	\end{equation}
	Now, combining Equations \eqref{s1} and \eqref{s2}, we get $L'(x)\leq 0,$ for $x\in\mathbb{R}^+.$ \\
	Case-II. $\lambda_1+\lambda_2\neq\mu_1+\mu_{2}.$ \\ In this case, we can take $\lambda_1+\lambda_2 = k(\mu_1+\mu_{2}),$ where $k$ is a scalar. Hence, $(k\mu_1,k\mu_2)\preceq^{m}(\lambda_1,\lambda_2)$. Let us consider $n^*$ dependent nonnegative random variables sharing Archimedean copula with generator $\psi,$ such that  $Z_{i}\sim F(k\mu_{1}x),$ for $i=1,\cdots,n^*_1$ and $Z_{j}\sim F(k\mu_{2}x),$ for $j={n^*_{1}+1},\cdots,n^*$.  Here, $n_{1}^{*}+n_{2}^{*}=n^*.$ Then, from  Case-I, we have $Z_{n^*:n^*}(n_{1}^{*},n_{2}^{*})\leq_{*}X_{n^*:n^*}(n_{1}^{*},n_{2}^{*})$. Further, star order is scale invariant, and hence we obtain $Y_{n^*:n^*}(n_{1}^{*},n_{2}^{*})\leq_{*}X_{n^*:n^*}(n_{1}^{*},n_{2}^{*}).$ This completes the proof of the theorem.
\end{proof}
%As an illustration of Theorem \ref{th}, we consider the following example.
%\begin{example}\label{exe1}
%	 Let $\boldsymbol{\lambda}=(7,0.9),~\boldsymbol{\mu}=(5.9,1),~(n^*_1,n^*_2)=(2,2),~\psi=e^{-x^\frac{1}{1.1}},~x>0.$ Also, let the baseline distribution be $F(x)=(\frac{x}{900})^{0.95},~0<x\leq 900.$
%	It can be shown that, all the conditions of Theorem \ref{th} are satisfied. Now, we plot the graphs of $f_{X_{4:4}}(x)$ and $f_{Y_{4:4}}(x)$  in Figure $2b$, which shows that the density function of $X_{4:4}$ is more skewed than that of $Y_{4:4}.$ This is consistent with the result in Theorem \ref{th}.
%\end{example}

Using the fact that the star order implies the Lorenz order, the following result is a direct consequence of Theorem \ref{th}. Further, since the Lorenz order is mainly used to compare the income distributions, the following corollary is more interesting from the point of its applications in the study of incomes.
\begin{corollary}\label{cor.th}
		Under the set-up as in Theorem \ref{th},
	$$\frac{\lambda_{2:2}}{\lambda_{1:2}}\geq\frac{\mu_{2:2}}{\mu_{1:2}}\Rightarrow Y_{n^{*}:n^{*}}(n^{*}_{1},n^{*}_{2})\leq_{Lorenz}X_{n^{*}:n^{*}}(n^{*}_{1},n^{*}_{2}),$$ provided
	$\frac{\psi}{\psi'}$ is decreasing, convex, $\frac{x\tilde{r}'(x)}{\tilde{r}(x)}$ is decreasing and  $x\tilde{r}(x)$ is increasing.
\end{corollary}
\subsection{Orderings between the smallest order statistics}
In the previous subsection, we focus on the conditions, under which the largest order statistics are comparable according to various stochastic orders. Here, we develop conditions such that the usual stochastic, hazard rate, star and Lorenz orders hold between the smallest order statistics. In the following theorems, we consider that the samples are heterogeneous and taken from the multiple-outlier dependent scale models. We now consider the following assumption.
\begin{assumption}\label{ass1.}
	Let $X_{1},\cdots,X_{n^{*}}$ be  $n^*$ dependent nonnegative random variables sharing  Archimedean survival copula with generator $\psi_1,$ with $X_{i}\sim F_1(x\lambda_1),$ for $i=1,\cdots,n^*_1$ and  $X_{j}\sim F_2(x\lambda_2),$ for $j={n^*_{1}+1},\cdots,n^*$. Also, 	let $Y_{1},\cdots,Y_{n^{*}}$ be  $n^*$ dependent non-negative random variables sharing Archimedean copula with generator $\psi_2,$ with $Y_{i}\sim F_1(x\mu_1),$ for $i=1,\cdots,n^*_1$ and $Y_{j}\sim F_2(x\mu_2),$ for $j={n^*_{1}+1},\cdots,n^*.$ Here, $n^{*}_{1}+n^{*}_{2}=n^{*}$, $\psi_{1}=\phi^{-1}_{1}$ and $\psi_{2}=\phi^{-1}_{2}$.
\end{assumption}

\begin{theorem}\label{th2}
	Under the set-up as in Assumption \ref{ass1.}, with  $r_1(x)\leq(\geq) r_2(x)$ and $n^{*}_1\leq(\geq)n^{*}_2,$
	 $$(\underbrace{\lambda_{1},\cdots,\lambda_{1}}_{n^{*}_{1}},\underbrace{\lambda_{2},\cdots,\lambda_{2}}_{n^{*}_{2}})\succeq_{w}
	(\underbrace{\mu_{1},\cdots,\mu_{1}}_{n^{*}_{1}},\underbrace{\mu_{2},\cdots,\mu_{2}}_{n^{*}_{2}})\Rightarrow X_{1:n^{*}}(n^{*}_{1},n^{*}_{2})\leq_{st}Y_{1:n^{*}}(n^{*}_1,n^{*}_2),$$ provided $\boldsymbol{ \lambda},~\boldsymbol{\mu}\in\mathcal{E_+}~(\mathcal{D_+}),$ $\phi_2\circ\psi_1$ is super-additive, $\psi_1\text{ or }\psi_2$ is log-convex and $r_1(x)\text{ or }r_2(x)$ is increasing.
\end{theorem}

\begin{proof}
	The reliability functions of $X_{1:n^{*}}(n^{*}_1,n^{*}_2)$ and $Y_{1:n^{*}}(n^{*}_1,n^{*}_2)$ are respectively given by
	$$\bar F_{X_{1:n^{*}}}(n^{*}_1,n^{*}_2)(x)=\psi_1\left[n^{*}_1\phi_1\left(\bar F_1\left(x\lambda_1\right)\right)+n^{*}_2\phi_1\left(\bar F_2\left(x\lambda_2\right)\right) \right]$$ and
	$$\bar F_{Y_{1:n^{*}}}(n^{*}_1,n^{*}_2)(x)=\psi_2\left[n^{*}_1\phi_2\left(\bar F_1\left(x\mu_1\right)\right)+n^{*}_2\phi_2\left(\bar F_2\left(x\mu_2\right)\right) \right].$$
	Let us denote $C(\boldsymbol{ \lambda},\psi_1,x)=\bar{F}_{X_{1:n^{*}}}(n^{*}_1,n^{*}_2)(x)$ and $D(\boldsymbol{ \mu},\psi_2,x)=\bar{F}_{Y_{1:n^{*}}}(n^{*}_1,n^{*}_2)(x).$ According to Lemma \ref{lem3.1}, super-additivity property of $\phi_2\circ\psi_1$ provides
	$C(\boldsymbol{ \mu},\psi_1,x)\leq D(\boldsymbol{ \mu},\psi_2,x).$ In order to prove the desired result, we need to show that
	$C(\boldsymbol{ \lambda},\psi_1,x)\leq C(\boldsymbol{ \mu},\psi_1,x).$ This is equivalent to show that $C(\boldsymbol{ \lambda},\psi_1,x)$ is decreasing and Schur-concave with respect to $\boldsymbol{ \lambda}$.
	Taking derivative of $C(\boldsymbol{ \lambda},\psi_1,x)$ with respect to $\lambda_1,$ we get
	\begin{equation}\label{eq4}
	\frac{\partial C(\boldsymbol{ \lambda},\psi_1,x)}{\partial\lambda_1}={-n^{*}_1xr_1(x\lambda_{1})}\frac{{\psi_{1}}\left[\phi_{1}\left(\bar{F_1}\left(x\lambda_1\right)\right)\right]}{{\psi}'_{1}\left[\phi_{1}\left(\bar{F_1}\left(x\lambda_1\right)\right)\right]}{\psi}'_{1}\left[n^{*}_1\phi_{1}\left(\bar{F_1}\left(x\lambda_1\right)\right)+n^{*}_2\phi_{1}\left(\bar{F_2}\left(x\lambda_2\right)\right)\right].
	\end{equation}
	Equation \eqref{eq4} shows that $C(\boldsymbol{ \lambda},\psi_1,x)$ is decreasing  in $\lambda_1,$ since $\frac{\partial C(\boldsymbol{ \lambda},\psi_1,x)}{\partial\lambda_1}\leq 0.$ Also,  $\frac{\partial C(\boldsymbol{ \lambda},\psi_1,x)}{\partial\lambda_2}\leq 0$. Therefore, $C(\boldsymbol{ \lambda},\psi_1,x)$ is decreasing in $\lambda_i,$ for $i=1,2.$ Further, to establish Schur-concavity of $C(\boldsymbol{ \lambda},\psi_1,x)$, we need to show that for $1\leq i\leq j\leq n^{*},$ the following inequality holds (see Lemma \ref{lem2.1b}~( Lemma \ref{lem2.1a})):
	\begin{equation}\label{eq3.4}
		\left[\frac{\partial C(\boldsymbol{ \lambda},\psi_1,x)}{\partial\lambda_i}-\frac{\partial C(\boldsymbol{ \lambda},\psi_1,x)}{\partial\lambda_j}\right]\geq(\leq) 0\text{, for }\boldsymbol{ \lambda}\in\mathcal{E}_+~(\mathcal{D}_+).
	\end{equation}
Now, we study the following cases.\\
	Case I: For $1\leq i\leq j\leq n^{*}_1 ,$  $\lambda_i=\lambda_j=\lambda_1.$ Thus, $\left[\frac{\partial C(\boldsymbol{ \lambda},\psi_1,x)}{\partial\lambda_i}-\frac{\partial C(\boldsymbol{ \lambda},\psi_1,x)}{\partial\lambda_j}\right]= 0.$\\
	Case II: For $n^{*}_{1}+1\leq i\leq j\leq n^{*} ,$  $\lambda_i=\lambda_j=\lambda_2.$ So,  $\left[\frac{\partial C(\boldsymbol{ \lambda},\psi_1,x)}{\partial\lambda_i}-\frac{\partial C(\boldsymbol{ \lambda},\psi_1,x)}{\partial\lambda_j}\right]= 0.$\\
	Case III: For $1\leq i\leq n^{*}_1 $ and $n^{*}_{1}+1\leq j\leq n^{*},$  $\lambda_i=\lambda_1$ and $\lambda_j=\lambda_2$.
	Suppose $\lambda_{1}\leq(\geq )\lambda_{2}$ which implies $\phi_{1}(\bar{F}_{1}(x\lambda_1))\leq(\geq)\phi_{1}(\bar{F}_{2}(x\lambda_2))$ in view of $F_1\leq(\geq) F_2.$  Now, applying the convexity property of $\ln\psi_{1}$,  we can write
	\begin{equation}\label{eq5}
	\frac{\psi_{1}(w)}{\psi'_{1}(w)}\Big|_{w=\phi_{1}[\bar{F_1}(x\lambda_{1})]}\geq(\leq) \frac{\psi_{1}(w)}{\psi'_{1}(w)}\Big|_{w=\phi_{1}[\bar{F_2}(x\lambda_{2})]}.
	\end{equation}
	Under the assumptions made, further, we obtain
	\begin{equation}\label{eq6}
	n^{*}_1r_{1}(x\lambda_{1})\leq(\geq)n^{*}_2 r_{2}(x\lambda_{2}).
	\end{equation}
	Finally, combining Equations \eqref{eq5} and \eqref{eq6}, we observe that the inequality in \eqref{eq3.4} holds. This completes the proof.	
\end{proof}

The following corollary is immediate from Theorem \ref{th2}.
\begin{corollary}\label{cor11}
	Let the set-up as in Assumption \ref{ass1.} hold with $\psi_1=\psi_2=\psi$ and $n^{*}_1\leq(\geq)n^{*}_2.$ Also, let $\psi$  be log-convex and $\boldsymbol{ \lambda},~\boldsymbol{\mu}\in\mathcal{E_+}~(\mathcal{D_+})$. Then,
	\begin{itemize}
			\item[(i)]
		 $(\underbrace{\lambda_{1},\cdots,\lambda_{1}}_{n^{*}_{1}},\underbrace{\lambda_{2},\cdots,\lambda_{2}}_{n^{*}_{2}})\succeq_{w}
		 (\underbrace{\mu_{1},\cdots,\mu_{1}}_{n^{*}_{1}},\underbrace{\mu_{2},\cdots,\mu_{2}}_{n^{*}_{2}})\Rightarrow X_{1:n^{*}}(n^{*}_{1},n^{*}_{2})\leq_{st}Y_{1:n^{*}}(n^{*}_1,n^{*}_2),$ provided
		  $r_1(x)\text{ or }r_2(x)$ is increasing with  $r_1(x)\leq(\geq) r_2(x)$.
		\item[(ii)]
		 $(\underbrace{\lambda_{1},\cdots,\lambda_{1}}_{n^{*}_{1}},\underbrace{\lambda_{2},\cdots,\lambda_{2}}_{n^{*}_{2}})\succeq_{w}
		 (\underbrace{\mu_{1},\cdots,\mu_{1}}_{n^{*}_{1}},\underbrace{\mu_{2},\cdots,\mu_{2}}_{n^{*}_{2}})\Rightarrow X_{1:n^{*}}(n^{*}_{1},n^{*}_{2})\leq_{st}Y_{1:n^{*}}(n^{*}_1,n^{*}_2),$ provided  $r(x)$ is increasing, where $r_1(x)=r_2(x)=r(x)$.
	\end{itemize}
\end{corollary}
The next result reveals that the smallest order statistics $X_{1:n^{*}}(n^{*}_{1},n^{*}_{2})$ dominates $X_{1:n}(n_{1},n_{2})$ in the sense of the usual stochastic order under the condition that $(n^{*}_1,n^{*}_2)$ is weakly submajorized by $(n_{1},n_{2}).$ The following assumption will be helpful to prove the next two results.
\begin{assumption}\label{ass3.}
	Let $X_{1},\cdots,X_{n^{*}}$ be $n^*$ dependent nonnegative random variables sharing Archimedean survival copula with generator $\psi_1,$ such that  $X_{i}\sim F_1(x\lambda_1)$,  for $i=1,\cdots,n^{*}_1$ and $X_{j}\sim F_{2}(x\lambda_{2}),$ for $j={n^{*}_{1}+1},\cdots,n^*.$ We assume that there exist two natural numbers $n_{1}$ and $n_{2}$ such that $1\leq n_{1}\leq n^{*}_{1}\leq n^{*}_{2}\leq n_{2}.$ Also,  $n=n_1+n_2,~n^*=n^*_1+n^*_2$ and $\psi_{1}=\phi^{-1}_{1}$.
\end{assumption}
\begin{theorem}\label{th5}
	Let Assumption \ref{ass3.} hold. Then, for  $\boldsymbol{\lambda}=(\lambda_1,\lambda_2)\in\mathcal{E_+}$ and $F_1\leq F_2,$ we have
	$$(n_{1},n_{2})\succeq_{w}(n^{*}_{1},n^{*}_{2})\Rightarrow X_{1:n}(n_{1},n_{2})\leq_{st}X_{1:n^{*}}(n^{*}_{1},n^{*}_{2}).$$
\end{theorem}

\begin{proof}
	To obtain the desired result, it is sufficient  to show that $\psi_{1}[\sum_{i=1}^{n}\phi_{1}\left(\bar{F_{i}}\left(x\lambda_i\right)\right)]\leq \psi_{1}[\sum_{i=1}^{n^{*}}\phi_{1}\left(\bar{F_{i}}\left(x\lambda_i\right)\right)]$.  Equivalently,
	\begin{eqnarray}\label{eq13}
	 (n^{*}_{1}-n_{1})\phi_{1}\left(\bar{F_{1}}\left(x\lambda_1\right)\right)\leq(n_{2}-n^{*}_{2})\phi_{1}\left(\bar{F_{2}}\left(x\lambda_2\right)\right).
	\end{eqnarray}
	Further, $(n_{1},n_{2})\succeq_{w}(n^{*}_{1},n^{*}_{2})\Rightarrow(n_{1}+n_{2})\geq(n^{*}_{1}+n^{*}_{2})\Rightarrow(n_{2}-n^{*}_{2})\geq(n^{*}_{1}-n_{1})\geq 0$ and $\lambda_{1}\leq \lambda_2\Rightarrow\phi_{1}\left(\bar{F_{2}}\left(x\lambda_2\right)\right)\geq\phi_{1}\left(\bar{F_{1}}\left(x\lambda_1\right)\right)\geq0.$
	Using these arguments, we get the inequality given by \eqref{eq13}. Hence, the proof is completed.
\end{proof}
In Theorem \ref{th5}, if we take the same baseline distribution functions, then we have the following corollary.
\begin{corollary}\label{cor12}
	Let Assumption \ref{ass3.} hold. Then, for  $\boldsymbol{\lambda}\in\mathcal{E_+},$ $F_1=F_2,$ we have
	$$(n_{1},n_{2})\succeq_{w}(n^{*}_{1},n^{*}_{2})\Rightarrow X_{1:n}(n_{1},n_{2})\leq_{st}X_{1:n^{*}}(n^{*}_{1},n^{*}_{2}).$$
\end{corollary}

In the following, we develop conditions such that the usual stochastic order holds between the smallest order statistics $X_{1:n}(n_{1},n_{2})$ and $Y_{1:n^{*}}(n^{*}_{1},n^{*}_{2})$.
The assumption in below is required for the next theorem.
\begin{assumption}\label{ass4.}
	Let $X_{1},\cdots,X_{n}$ be  $n$ nonnegative dependent random variables sharing Archimedean  survival copula with generator $\psi_1,$ such that  $X_{i}\sim F_1(x\lambda_{1}),$ for $i=1,\cdots,n_1$ and $X_{j}\sim F_2(x\lambda_{2}),$ for $j={n_{1}+1},\cdots,n.$ Also, let $Y_{1},\cdots,Y_{n^*}$ be  $n^*$ dependent nonnegative random variables sharing Archimedean copula with generator $\psi_2,$ such that  $Y_{i}\sim F_1(x\mu_{1}),$ for $i=1,\cdots,n^*_1$ and $Y_{j}\sim F_2(x\mu_{2}),$ for $j={n^*_{1}+1},\cdots,n^*.$  Here, $1\leq n_{1}\leq n^{*}_{1}\leq n^{*}_{2}\leq n_{2}$, $n=n_1+n_2$ and $n^*=n^*_1+n^*_2.$ 
\end{assumption}
\begin{theorem}\label{th6}
	Let Assumption \ref{ass4.} hold, with   $r_{1}(x)\leq r_{2}(x)$. Also, let $\boldsymbol{ \lambda},~\boldsymbol{\mu}\in\mathcal{E_+}.$ Then, for $(n_{1},n_{2})\succeq_{w}(n^{*}_{1},n^{*}_{2})$,
	 $$	 (\underbrace{\lambda_{1},\cdots,\lambda_{1}}_{n^{*}_{1}},\underbrace{\lambda_{2},\cdots,\lambda_{2}}_{n^{*}_{2}})\succeq_{w}(\underbrace{\mu_{1},\cdots,\mu_{1}}_{n^{*}_{1}},\underbrace{\mu_{2},\cdots,\mu_{2}}_{n^{*}_{2}})
\Rightarrow X_{1:n}(n_{1},n_{2})\leq_{st}Y_{1:n^{*}}(n^{*}_{1},n^{*}_{2}),$$ provided $\phi_2\circ\psi_1$ is super-additive,
	$\psi_1\text{ or }\psi_2$ is log-convex  and $r_1(x)\text{ or }r_2(x)$ is increasing.
\end{theorem}
 \begin{proof} The theorem can be proved using  Theorems \ref{th2} and \ref{th5}. Thus, it is omitted.
\end{proof}

As an illustration of Theorem \ref{th6}, we present the following example.
\begin{example}\label{ex3.4}
	Take $\boldsymbol{\lambda}=(2,6),~\boldsymbol{\mu}=(1,3),~(n_1,n_2)=(4,8),~(n^*_1,n^*_2)=(6,7),~\psi_1(x)=e^{-x^{\frac{1}{9}}}$ and $\psi_2(x)=e^{-x^{\frac{1}{10}}},~x>0.$ Also, let $F_1(x)=\left(\frac{x}{a}\right)^{l},~0<x\leq a$ and $F_2(x)=1-e^{-x},~x>0$.
	It can be seen that for $a=400$ and $l=2$ all the conditions of Theorem \ref{th6} are satisfied. Now, we plot the graph of $\bar{F}_{X_{1:12}}(4,8)(x)-\bar{F}_{Y_{1:13}}(6.7)(x)$, given in Figure $3a$. The figure suggests that $X_{1:12}(4,8)\leq_{st}Y_{1:13}(6,7)$ holds.
	\begin{figure}[h]
		\begin{center}
			\subfigure[]{\label{c2}\includegraphics[height=2.41in]{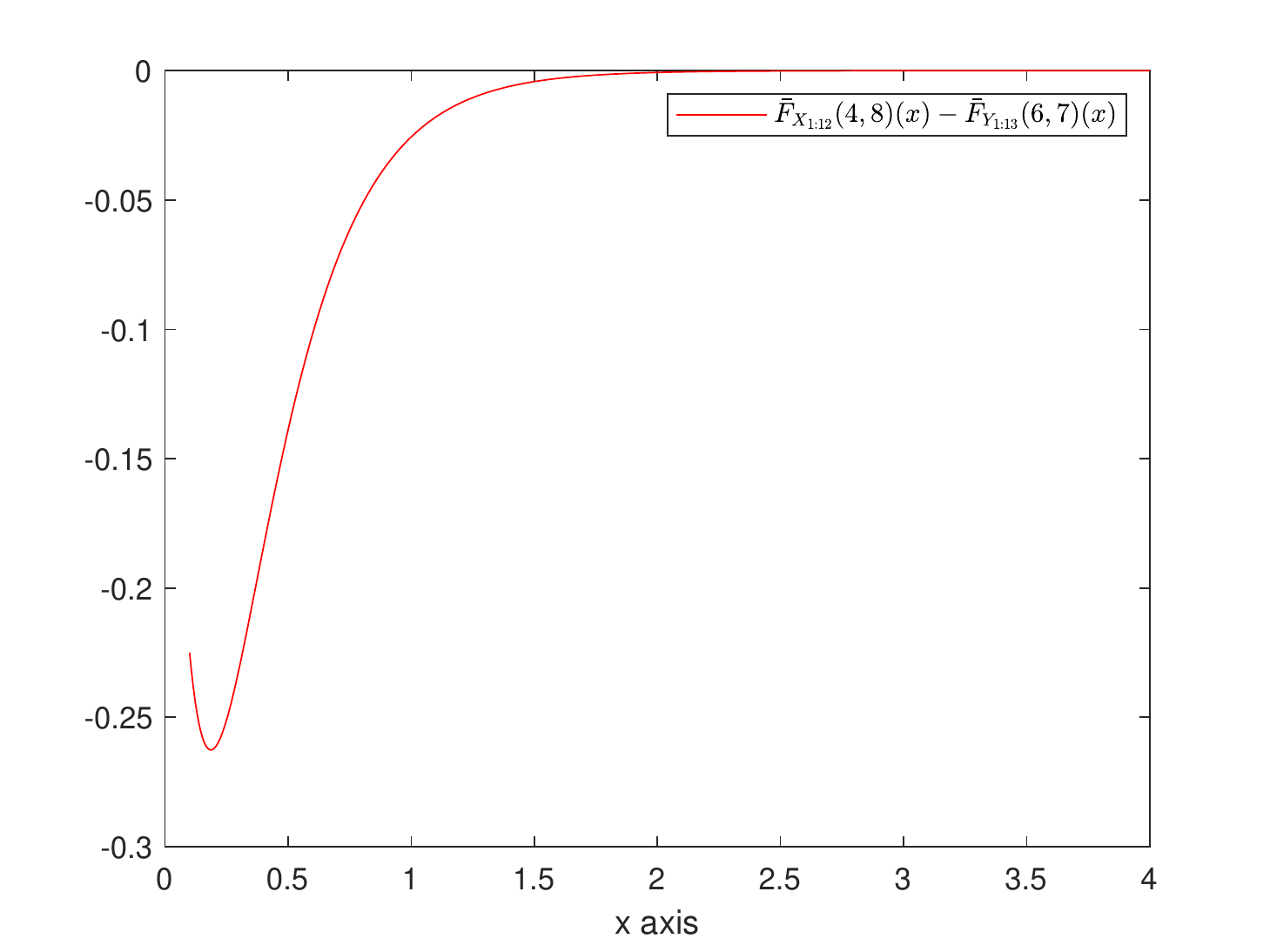}}
			\subfigure[]{\label{c1}\includegraphics[height=2.41in]{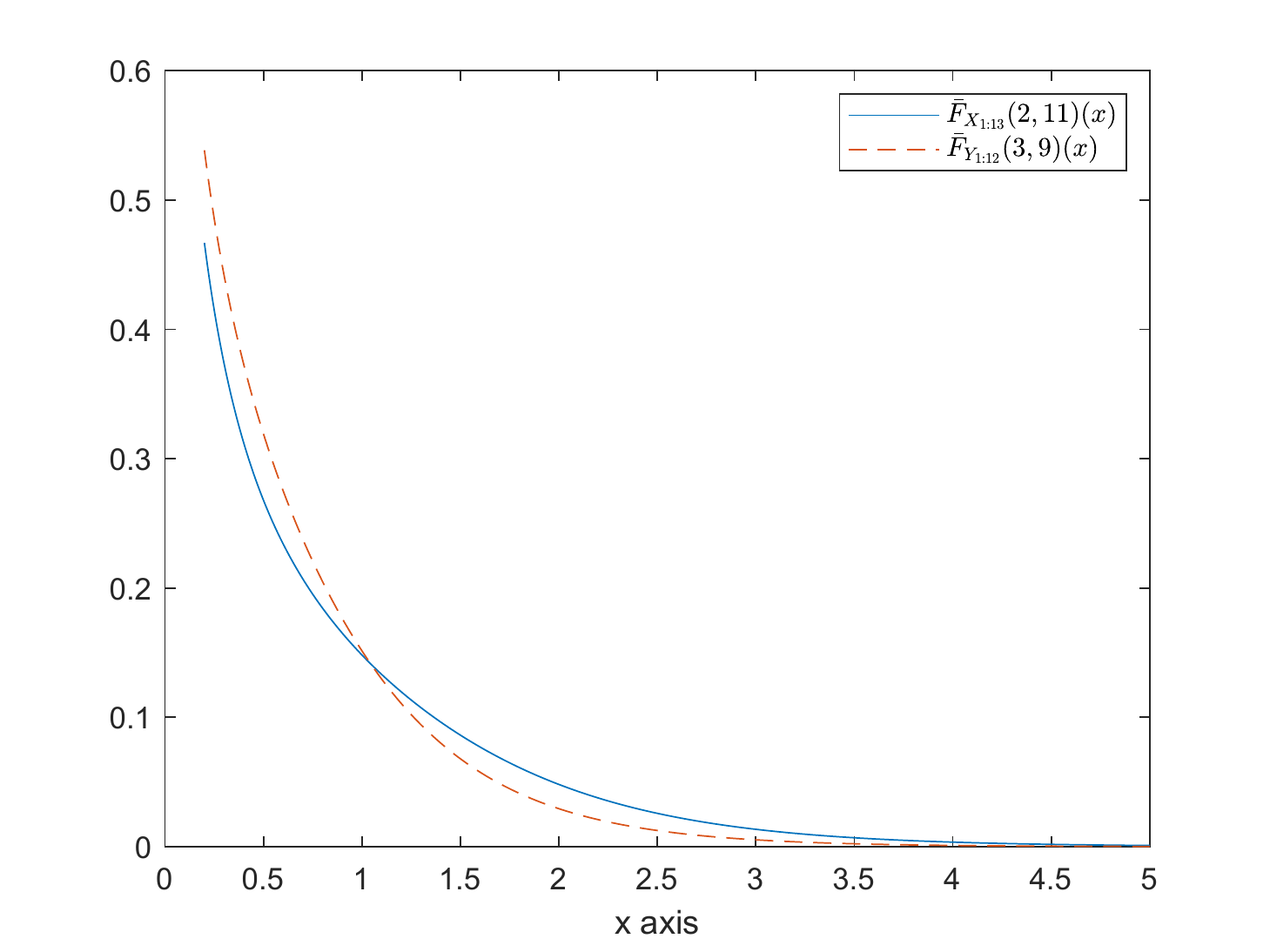}}
			\caption{
				(a) Plot of $\bar{F}_{X_{1:12}}(4,8)(x)-\bar{F}_{Y_{1:13}}(6,7)(x)$ as in Example \ref{ex3.4}. (b) Plots of $\bar{F}_{X_{1:13}}(2,11)(x)$ and $\bar{F}_{Y_{1:12}}(3,9)(x)$ as in Counterexample \ref{ce2}.
			}
		\end{center}
	\end{figure}	
\end{example}
Next, we present a counterexample, which shows that the stated usual stochastic order in Theorem \ref{th6} does not hold if the conditions $r_{1}(x)\leq r_{2}(x)$ and $r_2$ is increasing are dropped out.
\begin{counterexample}\label{ce2}
		Take $\boldsymbol{\lambda}=(1.2,3.6),~\boldsymbol{\mu}=(1.4,3),~(n_1,n_2)=(2,11),~(n^*_1,n^*_2)=(3,9),~\psi_1(x)=e^{-x^{\frac{1}{4.5}}}$ and $\psi_2(x)=e^{-x^{\frac{1}{5}}},~x>0.$ Also, suppose  $F_1(x)=1-e^{-x}$ and $F_2(x)=1-(1+2x)^{-0.5},~x>0.$
	Clearly, all the conditions of Theorem \ref{th6} are satisfied except $r_1\leq r_2$ and $r_2$ is increasing. Now, the graphs of $\bar{F}_{X_{1:13}}(2,11)(x)$ and $\bar{F}_{Y_{1:12}}(3,9)(x)$ are depicted in Figure $3b$. It reveals that the usual stochastic order in Theorem \ref{th6} does not hold.
\end{counterexample}
Upon using Theorem \ref{th6}, one can easily conclude the following corollary.
\begin{corollary}\label{cor13}
	Let Assumption \ref{ass4.} hold with $\psi_{1}=\psi_{2}=\psi$. Also, let $\boldsymbol{ \lambda},~\boldsymbol{\mu}\in\mathcal{E_+}$, 	$\psi$ is log-convex and
	$(n_{1},n_{2})\succeq_{w}(n^{*}_{1},n^{*}_{2})$. Then,
	\begin{itemize}
		\item [(i)]
		 $(\underbrace{\lambda_{1},\cdots,\lambda_{1}}_{n^{*}_{1}},\underbrace{\lambda_{2},\cdots,\lambda_{2}}_{n^{*}_{2}})\succeq_{w}(\underbrace{\mu_{1},\cdots,\mu_{1}}_{n^{*}_{1}},\underbrace{\mu_{2},\cdots,\mu_{2}}_{n^{*}_{2}})
		 \Rightarrow X_{1:n}(n_{1},n_{2})\leq_{st}Y_{1:n^{*}}(n^{*}_{1},n^{*}_{2}),$ provided $r_{1}(x)\text{ or }r_{2}(x)$ is increasing and $r_{1}(x)\leq r_{2}(x)$.
	\item [(ii)]
	 $(\underbrace{\lambda_{1},\cdots,\lambda_{1}}_{n^{*}_{1}},\underbrace{\lambda_{2},\cdots,\lambda_{2}}_{n^{*}_{2}})\succeq_{w}(\underbrace{\mu_{1},\cdots,\mu_{1}}_{n^{*}_{1}},\underbrace{\mu_{2},\cdots,\mu_{2}}_{n^{*}_{2}})
	 \Rightarrow X_{1:n}(n_{1},n_{2})\leq_{st}Y_{1:n^{*}}(n^{*}_{1},n^{*}_{2}),$ provided
	 $r(x)$ is increasing, where $r_1=r_2=r$.
	\end{itemize}
\end{corollary}

Next, we provide three consecutive theorems, which deal with the hazard rate ordering between the smallest order statistics.
\begin{theorem}\label{th12}
	Let Assumption \ref{ass1.} hold with $n^{*}_1\leq(\geq)n^{*}_2$, $\psi_1=\psi_2=\psi$, $F_1=F_2=F,$  $\tilde{r}_1=\tilde{r}_2=\tilde{r}$ and $r_1=r_2=r.$ Also, suppose $\psi$ is log-concave, $\frac{1-\psi}{\psi'}$ is decreasing, ${[\frac{1-\psi}{\psi'}]'}/{\frac{\psi}{\psi'}}$ is increasing and $\boldsymbol{ \lambda},~\boldsymbol{\mu}\in\mathcal{E_+}~(\mathcal{D_+})$. Then,
	 $$(\underbrace{m_{1},\cdots,m_{1}}_{n^{*}_{1}},\underbrace{m_{2},\cdots,m_{2}}_{n^{*}_{2}})\succeq_{w}(\underbrace{v_{1},\cdots,v_{1}}_{n^{*}_{1}},\underbrace{v_{2},\cdots,v_{2}}_{n^{*}_{2}})
	\Rightarrow X_{1:n^*}(n_{1}^*,n_{2}^*)\leq_{hr}Y_{1:n^*}(n^{*}_{1},n^{*}_{2}),$$
	where $m_{i}=\log\lambda_i$ and $v_{i}=\log\mu_i$, $i=1,~2,$ provided  $r(x)$ is increasing, $x\tilde r(x)$ is increasing and convex.
\end{theorem}

\begin{proof} Denote by $f$ the probability density function corresponding to the distribution function $F$. 
	The hazard rate function of $X_{1:n^{*}}(n^{*}_1,n^{*}_2)$ is given by
	\begin{eqnarray}\label{eq37}
	 r_{X_{1:n^{*}}}(n^{*}_1,n^{*}_2)(x)&\overset{def}{=}&\mathcal{E}(\boldsymbol{m})=\frac{\psi'\left[z\right]}{\psi\left[z\right]}\left[\frac{n^{*}_1e^{m_1}f(xe^{m_1})}{\psi'[\phi\left(\bar{F}\left(xe^{m_1}\right)\right)]}+\frac{n^{*}_2e^{m_2}f(xe^{m_2})}{\psi'[\phi\left(\bar{F}\left(xe^{m_2}\right)\right)]}\right]\nonumber\\
	& =&\frac{\psi'\left[z\right]}{\psi\left[z\right]}\Bigg[\frac{n^{*}_{1}e^{m_1}\tilde{r}(xe^{m_1})[1-\psi[\phi\left(\bar{F}\left(xe^{m_1}\right)\right)]]}{\psi'[\phi\left(\bar{F}\left(xe^{m_1}\right)\right)]}\nonumber\\
	&~&+\frac{n^{*}_{2}e^{m_2}\tilde{r}(xe^{m_2})[1-\psi[\phi\left(\bar{F}\left(xe^{m_2}\right)\right)]]}{\psi'[\phi\left(\bar{F}\left(xe^{m_2}\right)\right)]}\Bigg],
	\end{eqnarray}
	where ${z=n^{*}_{1}\phi\left(\bar{F}\left(xe^{m_1}\right)\right)}+n^{*}_{2}\phi\left(\bar{F}\left(xe^{m_2}\right)\right)$, $m_i=\log\lambda_i,$ for $i=1,2$ and $\boldsymbol{m}=(m_1,m_2).$ Also, $f$ is the probability density function of $F.$
	The partial derivative of $\mathcal{E}(\boldsymbol{m})$ with respect to $m_1$ is given by
	\begin{eqnarray}\label{eq38}
	\frac{\partial \mathcal{E}(\boldsymbol{m})}{\partial m_1}&=&-{n^{*}_{1}xe^{m_1}\tilde r(xe^{m_1})}\frac{d}{dz}\left[\frac{\psi'(z)}{\psi(z)}\right]\left[\frac{1-{\psi}\left[\phi\left[\bar{F}\left(xe^{m_1}\right)\right]\right]}{\psi'[\phi\left(\bar{F}\left(xe^{m_1}\right)\right)]}\right]\left[\sum_{i=1}^{n^{*}}\frac{e^{m_i}f\left(xe^{m_i}\right)}{\psi'[\phi\left(\bar{F}\left(xe^{m_i}\right)\right)]}\right]
	\nonumber\\
	&~&-n^{*}_{1}r(xe^{m_1})\left[x[e^{m_1}]^2\tilde {r}(xe^{m_1})\right]\frac{\psi'(z)}{\psi(z)}\left[[\frac{{\psi}(v)}{{\psi}'(v)}]^2\left[\frac{\frac{d}{dv}[\frac{1-{\psi}(v)}{{\psi}'(v)}]}{\frac{{\psi}(v)}{{\psi}'(v)}}\right]\right]_{v=\phi(\bar{F}(xe^{m_1}))}\nonumber\\
	&~&+n^{*}_{1}\frac{d}{dw}\left[w\tilde r(w)\right]_{w=xe^{m_1}}\frac{1-{\psi}\left[\phi\left[\bar{F}\left(xe^{m_1}\right)\right]\right]}{{\psi}'\left[\phi\left[\bar{F}\left(xe^{m_1}\right)\right]\right]}\frac{\psi'(z)}{\psi(z)}.
	\end{eqnarray}
	From \eqref{eq38}, it is easy to see that $\mathcal{E}(\boldsymbol{m})$ is increasing in $m_1.$ Similarly, $\mathcal{E}(\boldsymbol{m})$ is also increasing in $m_2.$ Hence, $\mathcal{E}(\boldsymbol{m})$ is increasing with respect to $\boldsymbol{m}.$ Now, we only need to show the Schur-convexity of $\mathcal{E}(\boldsymbol{m})$ with respect to $\boldsymbol{ m}.$ This is equivalent to show that
	for $1\leq i\leq j\leq n^{*},$
	\begin{equation}\label{eq3.5}
	\left[\frac{\partial \mathcal{E}(\boldsymbol{m})}{\partial m_i}-\frac{\partial \mathcal{E}(\boldsymbol{m})}{\partial  m_j}\right]\leq(\geq)0\text{, for }\boldsymbol{ m}\in\mathcal{E}_+~(\mathcal{D}_+).
	\end{equation}
	Utilizing the assumptions made, the rest of the proof follows from the similar arguments of Theorem \ref{th2}. Thus, it is omitted for the sake of conciseness.
\end{proof}
%The following result can be proved similar to Theorem \ref{th12}. Thus, the proof is not given.
%\begin{corollary}\label{cor14}
%	Let Set-up \ref{ass1.} hold with $\psi_1=\psi_2=\psi,$ $r_1=r_2=r$, $\tilde{r}_1=\tilde{r}_2=\tilde{r}$  and $F_1=F_2$. Then, for $\boldsymbol{ \lambda},~\boldsymbol{\mu}\in\mathcal{E_+}(~\mathcal{D_+})$
%	 $$(\underbrace{\log\lambda_{1},\cdots,\log\lambda_{1}}_{n_{1}},\underbrace{\log\lambda_{2},\cdots,\log\lambda_{2}}_{n_{2}})\succeq_{w}
%	(\underbrace{\log\mu_{1},\cdots,\log\mu_  {1}}_{n_{1}},\underbrace{\log\mu_{2},\cdots,\log\mu_{2}}_{n_{2}})\Rightarrow X_{1:n}(n_{1},n_{2})\leq_{hr}Y_{1:n}(n_{1},n_{2}),$$ provided
%	$\psi$ is log-concave, $\frac{1-\psi}{\psi'}$ is decreasing, ${[\frac{1-\psi}{\psi'}]'}/{\frac{\psi}{\psi'}}$ and $r(x)$ are increasing,  $x\tilde r(x)$ is increasing and convex.
%\end{corollary}
The following theorem demonstrates that under some conditions, the hazard rate ordering between  $X_{1:n}(n_{1},n_{2})$ and $X_{1:n^*}(n^{*}_{1},n^{*}_{2})$ exists.

\begin{theorem}\label{th13}
	Let Assumption \ref{ass3.} hold with $\psi_1=\psi$ and $\tilde{r}_{1}=\tilde{r}_2=\tilde{r}$. Then, for  $\boldsymbol{\lambda}\in\mathcal{E_+},$ we have
	$$(n_{1},n_{2})\succeq_{w}(n^{*}_{1},n^{*}_{2})\Rightarrow X_{1:n}(n_{1},n_{2})\leq_{hr} X_{1:n^{*}}(n^{*}_{1},n^{*}_{2}),$$
	provided $x \tilde{r}(x)$ is increasing, $\psi'/\psi$ and $\frac{1-\psi}{{\psi}'}$ are decreasing.
\end{theorem}
\begin{proof}
	The required result can be proved if we show that $r_{X_{1:n}}(n_1,n_2)(x)\geq 	r_{X_{1:n^{*}}}(n^*_1,n^*_2)(x)$ and equivalently,
	\begin{eqnarray}\label{eq44}
& \frac{\psi'\left[\sum_{i=1}^{n}\phi\left(\bar{F}\left(x\lambda_i\right)\right)\right]}{\psi\left[\sum_{i=1}^{n}\phi\left(\bar{F}\left(x\lambda_i\right)\right)\right]}\left[\sum_{i=1}^{n}\frac{\lambda_{i}\tilde{r}(x\lambda_{i})[1-\psi[\phi\left(\bar{F}\left(x\lambda_i\right)\right)]]}{\psi'[\phi\left(\bar{F}\left(x\lambda_i\right)\right)]}\right]\geq \frac{\psi'\left[\sum_{i=1}^{n^{*}}\phi\left(\bar{F}\left(x\lambda_i\right)\right)\right]}{\psi\left[\sum_{i=1}^{n^{*}}\phi\left(\bar{F}\left(x\lambda_i\right)\right)\right]}\left[\sum_{i=1}^{n^{*}}\frac{\lambda_{i}\tilde{r}(x\lambda_{i})[1-\psi[\phi\left(\bar{F}\left(x\lambda_i\right)\right)]]}{\psi'[\phi\left(\bar{F}\left(x\lambda_i\right)\right)]}\right].\nonumber\\
	\end{eqnarray}
	To prove inequality \eqref{eq44}, it is sufficient to show that the following two inequalities hold:
	\begin{eqnarray}\label{eq45}
	 (n^{*}_{1}-n_{1})\phi(\bar{F}(x\lambda_{1}))\leq  (n_{2}-n^{*}_{2})\phi(\bar{F}(x\lambda_{2}))
	\end{eqnarray}
	and
	\begin{eqnarray}\label{eq46}
	(n^{*}_1-n_{1})\frac{\lambda_{1}\tilde {r}(x\lambda_{1})[1-\psi[\phi\left(\bar{F}\left(x\lambda_1\right)\right)]]}{\psi'[\phi\left(\bar{F}\left(x\lambda_1\right)\right)]}\geq(n_{2}-n^{*}_{2})\frac{x\lambda_{2}\tilde {r}(x\lambda_{2})[1-\psi[\phi\left(\bar{F}\left(x\lambda_2\right)\right)]]}{\psi'[\phi\left(\bar{F}\left(x\lambda_2\right)\right)]}.
	\end{eqnarray}
	Further, $(n_{1},n_{2})\succeq_{w}(n^{*}_{1},n^{*}_{2})\Rightarrow(n_{1}+n_{2})\geq(n^{*}_{1}+n^{*}_{2})\Rightarrow(n_{2}-n^{*}_{2})\geq(n^{*}_{1}-n_{1})\geq 0.$ Also $\lambda_{1}\leq \lambda_2\Rightarrow\phi\left(\bar{F}\left(x\lambda_2\right)\right)\geq\phi\left(\bar{F}\left(x\lambda_1\right)\right)\geq 0.$
	By the help of decreasing property of $\frac{1-\psi}{{\psi}'}$, we obtain
	\begin{equation}\label{eq47}
	\frac{1-\psi(w)}{{\psi}'(w)}|_{w=\phi[\bar{F}(x\lambda_{2})]}\leq \frac{1-\psi(w)}{{\psi}'(w)}|_{w=\phi[\bar{F}(x\lambda_{1})]}\leq 0.
	\end{equation}
	Since $x \tilde{r}(x)$ is increasing,
	\begin{equation}\label{eq48}
	x\lambda_{1} \tilde{r}(x\lambda_{1})\leq x\lambda_{2} \tilde{r}(x\lambda_{2}).
	\end{equation}
	Thus, the proof is completed from Equations \eqref{eq47}, \eqref{eq48} and the given assumptions.
\end{proof}

The next theorem states that if the scale parameters are connected with the weakly submajorized order and the sample size pairs $(n_{1},n_{2})$ and $(n^{*}_{1},n^{*}_{2})$ have weakly submajorized order, then the smallest order statistics of $X_{1:n}(n_{1},n_{2})$ is dominated by $Y_{1:n^{*}}(n^{*}_{1},n^{*}_{2})$ according to the hazard rate order.
\begin{theorem}\label{th15}
Let Assumption \ref{ass4.} hold with $\psi_1=\psi_2=\psi,$ $r_1=r_2=r$ and $\tilde{r}_1=\tilde{r}_{2}=\tilde{r}$. Then, for  $\boldsymbol{ \lambda},~\boldsymbol{\mu}\in\mathcal{E_+}$ and
$(n_{1},n_{2})\succeq_{w}(n^{*}_{1},n^{*}_{2})$,
$$(\underbrace{m_{1},\cdots,m_{1}}_{n^{*}_{1}},\underbrace{m_{2},\cdots,m_{2}}_{n^{*}_{2}})\succeq_{w}(\underbrace{v_{1},\cdots,v_{1}}_{n^{*}_{1}},\underbrace{v_{2},\cdots,v_{2}}_{n^{*}_{2}})\Rightarrow X_{1:n}(n_{1},n_{2})\leq_{hr}Y_{1:n^*}(n^{*}_{1},n^{*}_{2}),$$ provided
$\psi$ is log-concave, $\frac{1-\psi}{\psi'}$ is decreasing, ${[\frac{1-\psi}{\psi'}]'}/{\frac{\psi}{\psi'}}$ and $r(x)$ are increasing, $x\tilde r(x)$ is increasing and convex, where $m_i=\log\lambda_i$ and  $v_i=\log\mu_i$, $i=1,~2$.
\end{theorem}
\begin{proof}
	The proof of the theorem follows from Theorem \ref{th12} and Theorem \ref{th13}. Thus, it is omitted.
\end{proof}
%The proof of the next corollary is same as Theorem \ref{th15}. Hence, we will not discuss the proof of the following corollary.
%\begin{corollary}\label{cor15}
%	Let Set-up \ref{ass4.} hold with $\psi_1=\psi_2=\psi,$ $F_1=F_2,$ $r_1=r_2=r$ and $\tilde{r}_1=\tilde{r}_{2}=\tilde{r}$. Then, for  $\boldsymbol{ \lambda},~\boldsymbol{\mu}\in\mathcal{D_+},$
%	$(n_{1},n_{2})\succeq^{w}(n^{*}_{1},n^{*}_{2})$ and
%	 $$(\underbrace{\log\lambda_{1},\cdots,\log\lambda_{1}}_{n^{*}_{1}},\underbrace{\log\lambda_{2},\cdots,\log\lambda_{2}}_{n^{*}_{2}})\succeq_{w}
%	 (\underbrace{\log\mu_{1},\cdots,\log\mu_{1}}_{n_{1}},\underbrace{\log\mu_{2},\cdots,\log\mu_{2}}_{n_{2}})\Rightarrow X_{1:n^*}(n^{*}_{1},n^{*}_{2})\leq_{hr}Y_{1:n}(n_{1},n_{2}),$$ provided
%	$\psi$ is log-concave, $\frac{1-\psi}{\psi'}$ is decreasing, ${[\frac{1-\psi}{\psi'}]'}/{\frac{\psi}{\psi'}}$ and $r(x)$ are increasing, $x\tilde r(x)$ is increasing and convex .
%\end{corollary}
To illustrate Theorem \ref{th15}, we now consider the following example.
\begin{example}\label{ex3.5}
	Set $\boldsymbol{\lambda}=(e^{0.5},e^{0.6}),~\boldsymbol{\mu}=(e^{0.2},e^{0.3}),~(n_1,n_2)=(2,11),~(n^*_1,n^*_2)=(3,7)$ and $\psi(x)=e^{\frac{1}{0.99}(1-e^{x})},~x>0.$ Further, let $F(x)=\left(\frac{x}{a}\right)^{l},~0<x\leq a.$
	It can be easily shown that for $a=1000$ and $l=2,$ all the conditions of Theorem \ref{th15} are satisfied. Now, we plot the ratio $\frac{\bar{F}_{Y_{1:10}}(3,7)(x)}{\bar{F}_{X_{1:13}}(2,11)(x)}$  in Figure $2b$, which is consistent with the result in Theorem \ref{th15}.
%	\begin{figure}[h]
%		\begin{center}
%			\subfigure[]{\label{c5}\includegraphics[height=2.41in]{M_Oth3_13.eps}}
%			\subfigure[]{\includegraphics[height=2.41in]{M_Oth3_13_star.eps}}
%			%	\subfigure[]{\label{c6}\includegraphics[height=2.41in]{M_Oth3_13.eps}}
%			\caption{
%				(a) Plot of $\bar{F}_{Y_{1:10}}(3,7)(x)/\bar{F}_{X_{1:13}}(2,11)(x)$ as in Example \ref{ex3.5}.
%				(b) Graphs of $f_{X_{1:4}}(x)$ and $f_{Y_{1:4}}(x)$ for Example \ref{exe2}.
%				%				 (b) Plot of $\frac{\bar{F}_{Y_{1:10}}(3,7)(x)}{\bar{F}_{X_{1:13}}(2,11)(x)}.$
%			}
%		\end{center}
%	\end{figure}
\end{example}

In the next theorem, we develop some conditions under which two smallest order statistics are comparable according to the star order.
\begin{theorem}\label{th.}
	Under the set-up as in Assumption \ref{ass1.}, with $\tilde{r}_1=\tilde{r}_2=\tilde{r}$ and $\psi_1=\psi_2=\psi$,
	$$\frac{\lambda_{2:2}}{\lambda_{1:1}}\geq\frac{\mu_{2:2}}{\mu_{1:2}}\Rightarrow Y_{1:n^{*}}(n^{*}_{1},n^{*}_{2})\leq_{*}X_{1:n^{*}}(n^{*}_{1},n^{*}_{2}),$$ provided
	$\frac{\psi}{\psi'}$ is decreasing, convex, $\frac{xr'(x)}{r(x)}$ and  $xr(x)$ are decreasing.
\end{theorem}

\begin{proof}
	The distribution functions of $X_{1:n^{*}}(n^{*}_1,n^{*}_2)$ and $Y_{1:n^{*}}(n^{*}_1,n^{*}_2)$ are respectively given by
	$$F_{X_{1:n^{*}}}(n^*_1,n^*_2)(x)=1-\psi\left[n^{*}_1\phi\left(\bar {F}\left(x\lambda_1\right)\right)+n^{*}_2\phi\left(\bar {F}\left(x\lambda_2\right)\right) \right]$$ and
	$$F_{Y_{1:n^{*}}}(n^*_1,n^*_2)(x)=1-\psi\left[n^{*}_1\phi\left(\bar {F}\left(x\mu_1\right)\right)+n^{*}_2\phi\left(\bar {F}\left(x\mu_2\right)\right) \right].$$
Now, the rest of the proof follows using similar arguments as in Theorem \ref{th}. Thus, it is omitted.
\end{proof}
%The following example is a demonstration of Theorem \ref{th}.
%\begin{example}\label{exe2}
%		Let us consider $\boldsymbol{\lambda}=(0.8,5),~\boldsymbol{\mu}=(2.1,3.7),~(n^*_1,n^*_2)=(2,2)$ and $\psi(x)=e^{-x^\frac{1}{1.1}},~x>0.$ Also, let  $F(x)=1-x^{-0.9},~x\geq1.$
%	Under this set-up, all the conditions of Theorem \ref{th.} are satisfied. Now, we plot the graph of $f_{X_{1:4}}(x)$ and $f_{Y_{1:4}}(x)$  in Figure $4b$, which reveals that the density function of $X_{1:4}$ is more skewed than $Y_{1:4},$ which is consistent with the result in Theorem \ref{th.}.
%\end{example}

The following result is a direct consequence of Theorem \ref{th.}.
\begin{corollary}
		Under the assumptions as in Theorem \ref{th.},
	$$\frac{\lambda_{2:2}}{\lambda_{1:1}}\geq\frac{\mu_{2:2}}{\mu_{1:2}}\Rightarrow Y_{1:n^{*}}(n^{*}_{1},n^{*}_{2})\leq_{Lorenz}X_{1:n^{*}}(n^{*}_{1},n^{*}_{2}),$$ provided
	$\frac{\psi}{\psi'}$ is decreasing, convex, $\frac{xr'(x)}{r(x)}$ and  $xr(x)$ are decreasing.
\end{corollary}
%\section{Application of Archimedean copula}
%In Table \ref{}, we have listed a number of well-known Copulas that could play a role of the baseline distribution for a
%location-scale family. We will discuss now various relevant properties of all these distributions to be used for the meaningful
%illustration of the main results of this paper. We believe that the forthcoming results have their own value for distribution
%theory as well. We begin with the following lemma which will be used in the next proposition.
\section{Concluding remarks}
Due to simplicity in tackling the expressions/terms, most of the researchers have concentrated on the multiple-outlier models and studied ordering results between the order statistics under the set-up of independent random variables. However, the assumption of independent random variables is not feasible in many situations. So, it is required to assume dependent structure among the random observations. In this paper, we discussed some comparison results between the lifetimes of both parallel and series systems consisting of multiple-outlier dependent scale components in the sense of the usual stochastic, reversed hazard rate, hazard rate, star and Lorenz orders. The dependence structure has been modeled by Archimedean copulas. Sufficient conditions have been established for the purpose of the comparisons of the order statistics. Several examples and counterexamples are presented to illustrate the established results.
\\
\\
\\
{\large\bf Acknowledgements:} Sangita Das thanks the
financial support provided by the MHRD, Government of India. Suchandan Kayal gratefully acknowledges the partial financial support for this research work under a grant MTR/2018/000350, SERB, India.

\bibliography{ref}

\end{document}